\documentclass[10pt]{amsart}
\usepackage{amsmath}
  \usepackage{graphics} 
\usepackage{graphicx}  
  \usepackage{epsfig} 
 \usepackage[colorlinks=true]{hyperref}
  \usepackage{float}
  \usepackage{caption}
\usepackage{amsmath}
\usepackage{wrapfig}
\usepackage{graphicx}
\usepackage{subcaption}
\usepackage{amssymb}

\usepackage{color}
\usepackage{enumitem} 
\usepackage{comment}

\newtheorem{theorem}{Theorem}[section]
\newtheorem{corollary}{Corollary}

\newtheorem{lemma}[theorem]{Lemma}
\newtheorem{proposition}{Proposition}

\theoremstyle{definition}

\newtheorem{remark}{Remark}

\newcommand{\beq}{\begin{equation*}}
\newcommand{\eeq}{\end{equation*}}
\newcommand{\beqn}{\begin{eqnarray*}}
\newcommand{\eeqn}{\end{eqnarray*}}

\title[ TYC strategy ] 
      {A Comparison of the Trojan Y Chromosome Strategy to Harvesting Models for Eradication of Non-Native Species}


\begin{document}
\maketitle

\centerline{
\scshape  Jingjing Lyu$^{1}$, Pamela J. Schofield$^{2}$, Kristen M. Reaver$^{3}$, }
\centerline{
\scshape Matthew Beauregard$^{4}$ and Rana D. Parshad$^{5}$ 
 }
\medskip
{\footnotesize
 \centerline{1) Department of Mathematics,}
 \centerline{Chengdu University,}
   \centerline{ Chengdu,  Sichuan 610106, China.}
   \medskip
   \centerline{ 2) U.S. Geological Survey,}
 \centerline{ Wetland and Aquatic Research Center,}
   \centerline{Gainesville, FL  32653, USA.}
    \medskip
      \centerline{3) Cherokee Nation Technologies,}
 \centerline{Wetland and Aquatic Research Center,}
   \centerline{Gainesville, FL 32653, USA}
     \medskip
   \centerline{4) Department of Mathematics,}
 \centerline{Stephen F. Austin State University,}
  \centerline{Nacogdoches, TX 75962 , USA }
    \medskip
   \centerline{5) Department of Mathematics,}
 \centerline{Iowa State University,}
   \centerline{Ames, IA 50011, USA}
 
 }

\noindent
\quad \\
\quad \\


\begin{abstract}
The Trojan Y Chromosome Strategy (TYC) is a promising eradication method for biological control of non-native species. The strategy works by manipulating the sex ratio of a population through the introduction of \textit{supermales} that guarantee male offspring. In the current manuscript, we compare the TYC method with a pure harvesting strategy.  We also analyze a hybrid harvesting model that mirrors the TYC strategy. The dynamic analysis leads to results on stability of solutions and bifurcations of the model. Several conclusions about the different strategies are established via optimal control methods. In particular, the results affirm that either a pure harvesting or hybrid strategy may work better than the TYC method at controlling a non-native species population.
\end{abstract}
 \section*{Recommendations for Resource Managers}
 
 \begin{itemize}

 \item Where harvesting is feasible, it is as effective if not more effective than the classical TYC method. 
 Therein managers may attempt harvesting female fish while stocking males or harvesting both male and female fish.

\item Managers may attempt linear harvesting, saturating density dependent harvesting and unbounded density dependent harvesting. Linear harvesting is seen to be the most effective.

\item We caution against the outright use of harvesting due to various density dependent effects that may arise. To this end hybrid models that involve a combination of harvesting and TYC type methods might be a better strategy.  

 \item One may also use harvesting as a tool in mesocosm settings to predict the efficacy of the TYC strategy in the wild.

 \end{itemize}

 \section{Introduction}
\subsection{Background}

Biological invasions are the ``uncontrolled spread and proliferation of species to areas outside their native range" \cite{L16}. The rate of such invasions in the United States continues to rise and, subsequently, the financial and ecological damage caused by them \cite{Pimentel05, S00}. Non-native species can be difficult to manage \cite{2,14}. For such reasons, the spread and control of non-native species is an important and timely problem in spatial ecology and much work has been devoted to this issue \cite{A06, A12, B07, C01, L12, M00, 089, S97, V96}. Current eradication efforts for invasive aquatic species usually involve chemical treatment, local harvesting,  dewatering, ichthyocides, or a suitable combination \cite{SL15}. Unfortunately, all of these methods are known to negatively impact native fauna \cite{SL15}.

An \emph{alternative} method that has been proposed for eradication of non-native species is the Trojan Y Chromosome strategy (TYC) \cite{GutierrezTeem06, TGP13}. Unlike other technological approaches, the TYC strategy does not require within-chromosome genetic modifications, rather, it involves a reassortment of (whole) pre-existing sex chromosomes among individuals and is not considered a genetically modified organism (GMO) \cite{Schill2017}. These manipulations can cease at any time and therefore the strategy is reversible. The strategy works by adding feminized males and/or feminized super males (containing two Y chromosomes) to an existing invasive population, to skew the sex ratio of subsequent generations to contain an increasing number of males (i.e., fewer and fewer females in each generation). 
The gradual reduction in females may lead to eventual extinction of the population (see the right panel of Fig.~\ref{Fig:TYC-Eradication-Strategy_comp}). This strategy has been of much interest lately \cite{TGP13, CW07a, CW07b, CW09, Gutierrez12, P11, p10, ParshadGutierrez11, Parshad13, P09,  SDE2013, ODE2013, Z12}. 
%
\begin{figure}[H]
 \centering{
  {\includegraphics[height=4.5cm, width=6cm]{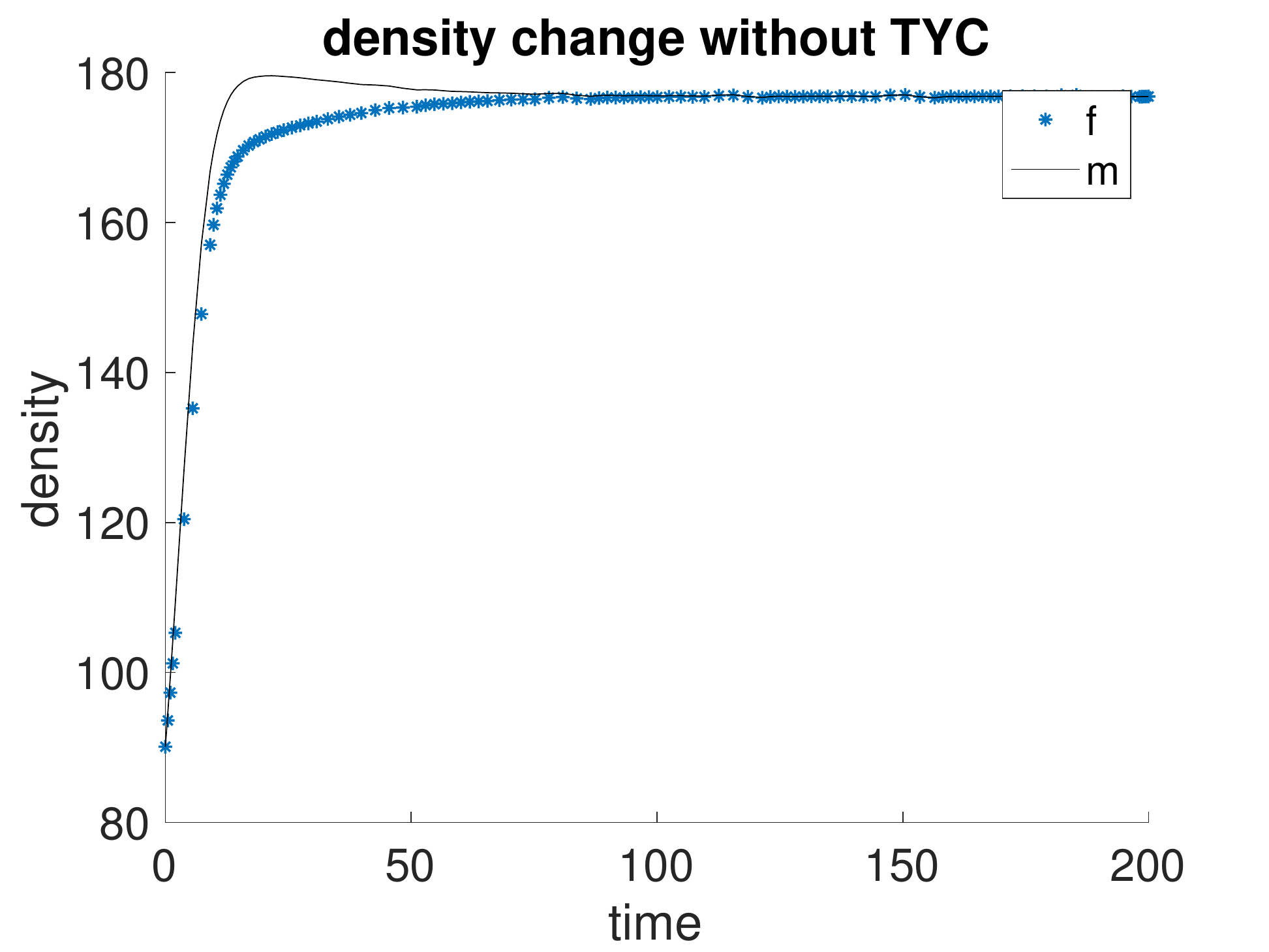}
  \includegraphics[height=4.5cm, width=6cm]{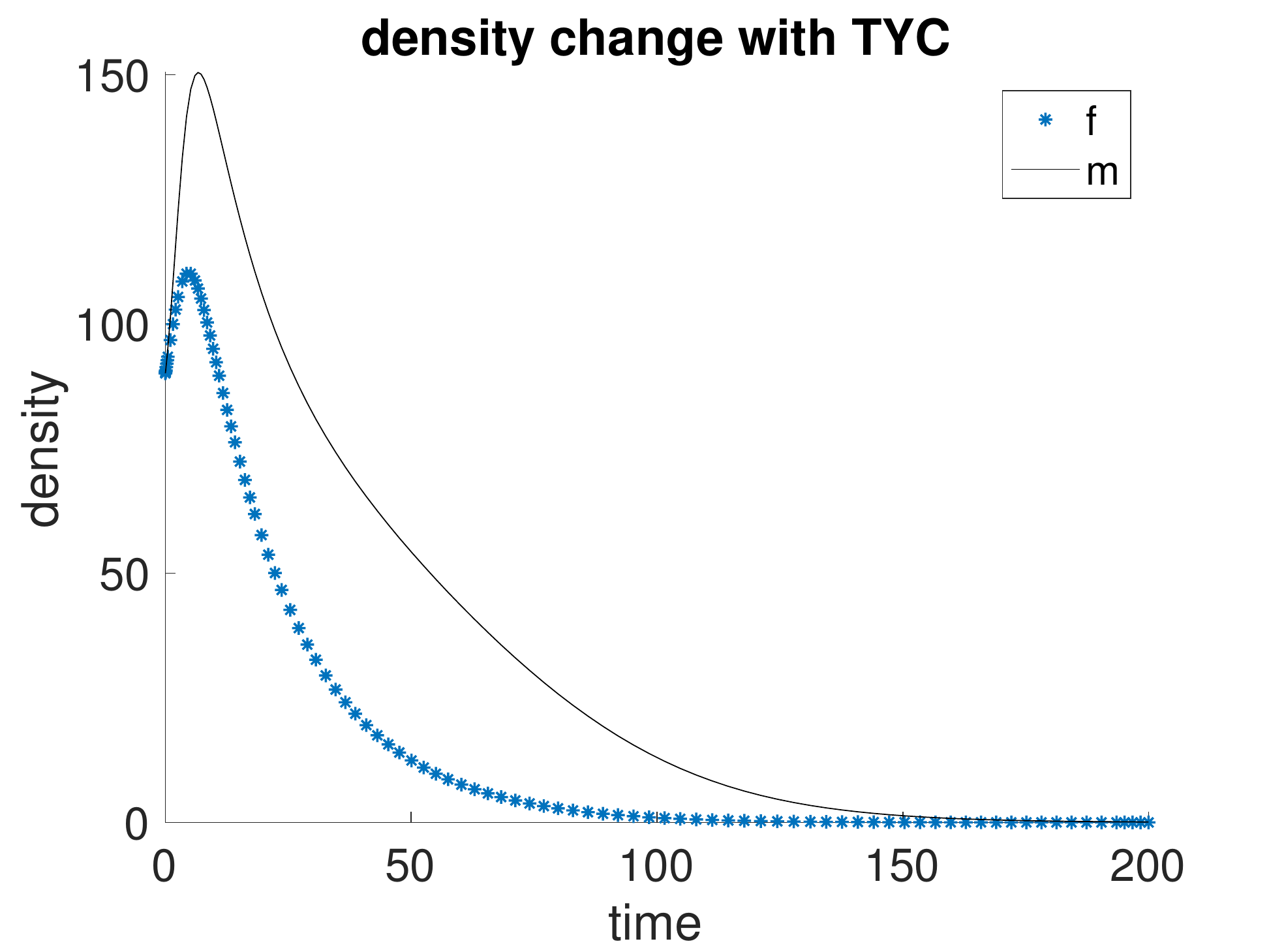}
  
  }
 \caption{\small Left panel: the female/male density change with time without TYC strategy. Right Panel: the female/male density change with time by introducing YY supermales (TYC strategy).}
 \label{Fig:TYC-Eradication-Strategy_comp}
 }
 \end{figure}
 

Fig.~\ref{Fig:TYC-Eradication-Strategy_comp} is a simple demonstration of the theoretical power of the TYC strategy. There is recent laboratory and field progress on TYC from two groups. (1) Dr. Pamela Schofield's laboratory, at the U.S. Geological Survey Wetland (USGS) and Aquatic Research Center in Gainesville, Florida is working on developing YY males for several fishes, including the fancy guppy fish (\textit{Poecilia reticulata}). (2) The Eagle Fish Genetics Lab of the Idaho Department of Fish and Game has produced large stocks of YY males, for brook trout (\emph{Salvelinus fontinalis}) \cite{Schill16}. These have been released in the field with promising preliminary results \cite{Schill18}. However, personal communication with chief scientists at the above labs \cite{SchillPersonalComm} suggests that the challenges in putting TYC into practice are (1) the design and production of the $YY$ males and females and (2) \emph{in vitro} testing of the strategy. The goal of this manuscript is to compare harvesting strategies with TYC. We wanted to know whether TYC, a new technology, could out compete the traditional strategy of harvesting in reducing populations of non-native species. Harvesting has been used to reduce populations of non-native species and subsequently moderate their negative effects on environments \cite{G14, P96}. The use of harvesting is restricted to certain field situations where the target organisms can be encountered, detected and removed in a practical manner.

Our specific objectives were:
\begin{enumerate}
\item Investigate alternate biological control strategies that \emph{do not} require $YY$ males or females. 
\item Compare and contrast such strategies to the TYC strategy, via an optimal control approach. 
\item Develop such strategies so that they could be used in \emph{conjunction} with the established TYC strategy as a hybrid strategy or become a \emph{novel} strategy in itself.
\item To use preliminary population data from mesocosm experiments to establish our models. Note that these experiments contain wild type males and females only. There are no YY males present in the mesocosm.
\item To compare and contrast our harvesting strategies to the TYC strategy, using realistic parameters that are \emph{outputs of the above mentioned mesocosm experiments}.
\item Establish a mesocosm framework via harvesting - that could in turn be used to test the efficacy of the TYC strategy in the wild.
\end{enumerate}
 
For our mesocosm experiments we have chosen to work with guppy (\emph{Poecilia reticulata}) in the laboratory and get the data to model its population dynamics herein. Our experiments contain only wild type male and female guppies. Guppy is a tropical ornamental fish that is popular as an aquarium pet. This species is introduced intermittently across the USA and has established local reproducing populations in some areas \cite{N19}. Guppy is used widely in laboratory studies and is especially well-suited to our experiments due to its short generation time (ca. four weeks until sexual maturity), sexual dimorphism, and peaceful nature.


\subsection{Three-Variables TYC Model}
The three-variables TYC model, in which only a YY supermale is introduced, is described by a system of three ordinary differential equations for state variables: a wild-type XX female ($f$), a wild-type XY male ($m$) and a YY supermale ($s$):
\begin{equation}
\label{TYC_1}
\begin{split}
& \frac{df}{dt}=\frac{1}{2}fm\beta L-\delta f,\\
& \frac{dm}{dt}=\left(\frac{1}{2}fm+fs\right)\beta L-\delta m,\\
& \frac{ds}{dt}=\mu-\delta s,
\end{split}
\end{equation}
where $f,~m,$ and $s$ are the population densities; $\beta$ represents the birth rate; $\delta$ is the death rate; $\mu$ denotes the non-negative introduction rate of YY supermales. The logistic term is given by
\begin{equation}
\label{jj_1}
L=1-\frac{f+m+s}{K},\quad\quad\quad\quad\quad\quad\quad\quad\quad
\end{equation}
where $K$ is the carrying capacity of the ecosystem. It is assumed that $0<\beta$ and $\delta<1$ in this manuscrript. 

\subsection{Preliminary Data}
Eleven months of preliminary population data were collected from a mesocosm experiment conducted at the USGS facility in Gainesville, Florida (Fig.~\ref{Fig:fig2n1_jj}, Fig.~\ref{Fig:fig2n1}). Note these experiments consisted of wild type males and wild type females only. There are no YY males present in the mesocosm. Our goal was to infer life history parameters for \emph{P. reticulata}, such as birth and death rates and carrying capacity, from this data - and then use that to simulate TYC type models. The experiment was conducted in an indoor laboratory in one rectangular fiberglass tank (114 x 56 x 61 cm) with aerated well water at a depth of 30 cm (192L of total volume). Period water temperature fluctuated with ambient indoor temperature  and ranged from approximately 22-25 C.  Living aquatic vegetation (Hydrilla) was added to provide cover for the fish and to seed live food sources such as microcrustaceans. The fish were fed three times weekly with commercial flake food to supplement the live food sources. No predators or inter-specific competitors were present in the mesocosms. A total of 30 wild-type fancy guppies (15 XY males and 15 XX females) were initially introduced to the mesocosm. They were allowed to reproduce in the mesocosm for 11 months. Population counts were made monthly (note that the guppy generation time is $\approx$ 1 month) - see right panel in Fig.~\ref{Fig:fig2n1}. 
 
The population data is best fit to the basic population model where $s=0$ (without introduced YY super males) in \eqref{TYC_1}. This enables the various population parameters to be inferred, see right panel in Fig.~\ref{Fig:fig2n1}. The best fit parameters are $\beta=0.0057, ~\delta=.0.0648,$ and $K=405.$
 
 \begin{figure}[h!]
 \centering{
  \includegraphics[scale=0.18]{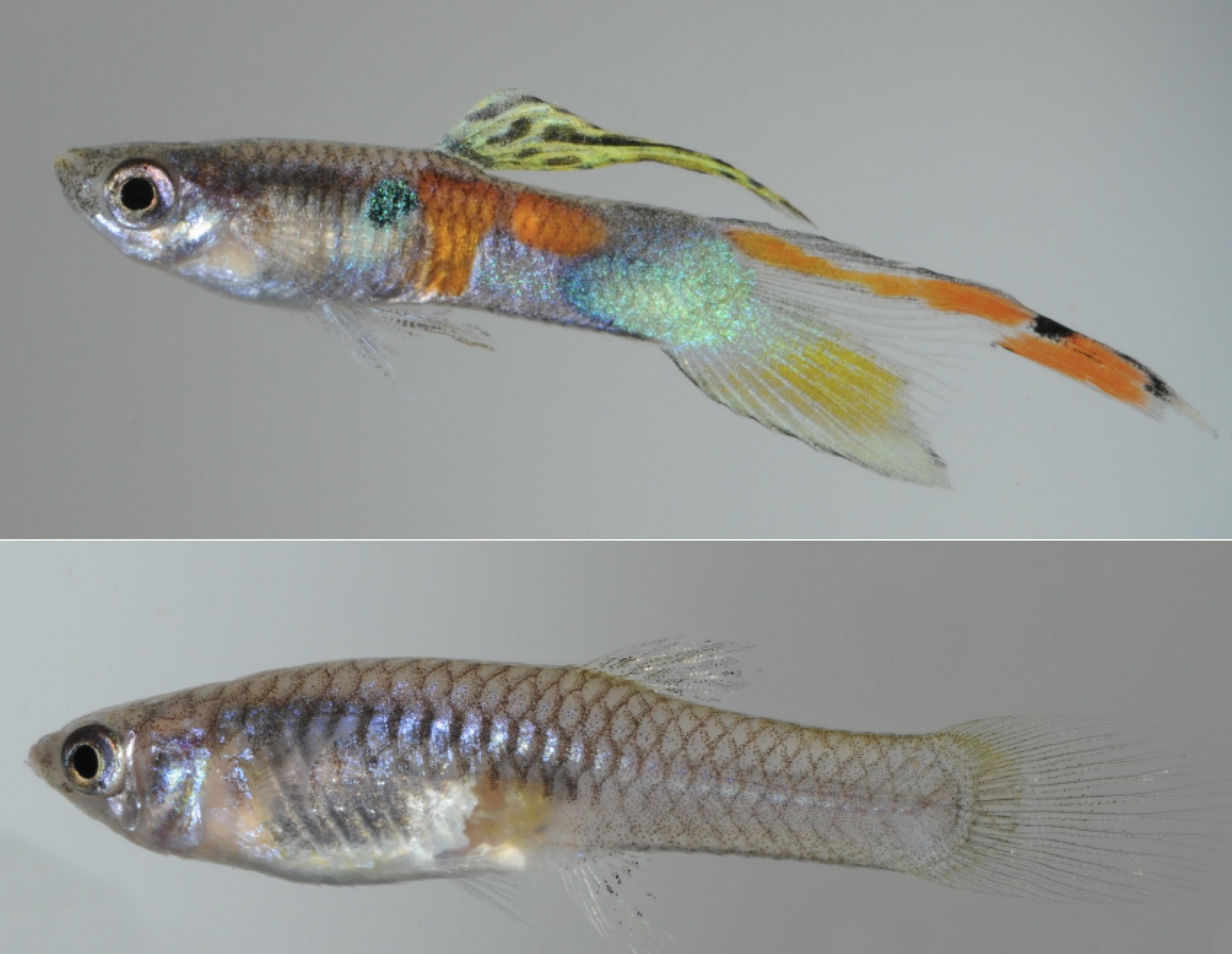}
  }
 \caption{\small Male (top) and female (bottom) fancy guppy (\textit{Poecilia reticulata}). Photos were taken by Howard Jelk, USGS.
 \label{Fig:fig2n1_jj} 
 }
 \end{figure}

 \begin{figure}[H]
 \centering{
  \includegraphics[scale=0.035]{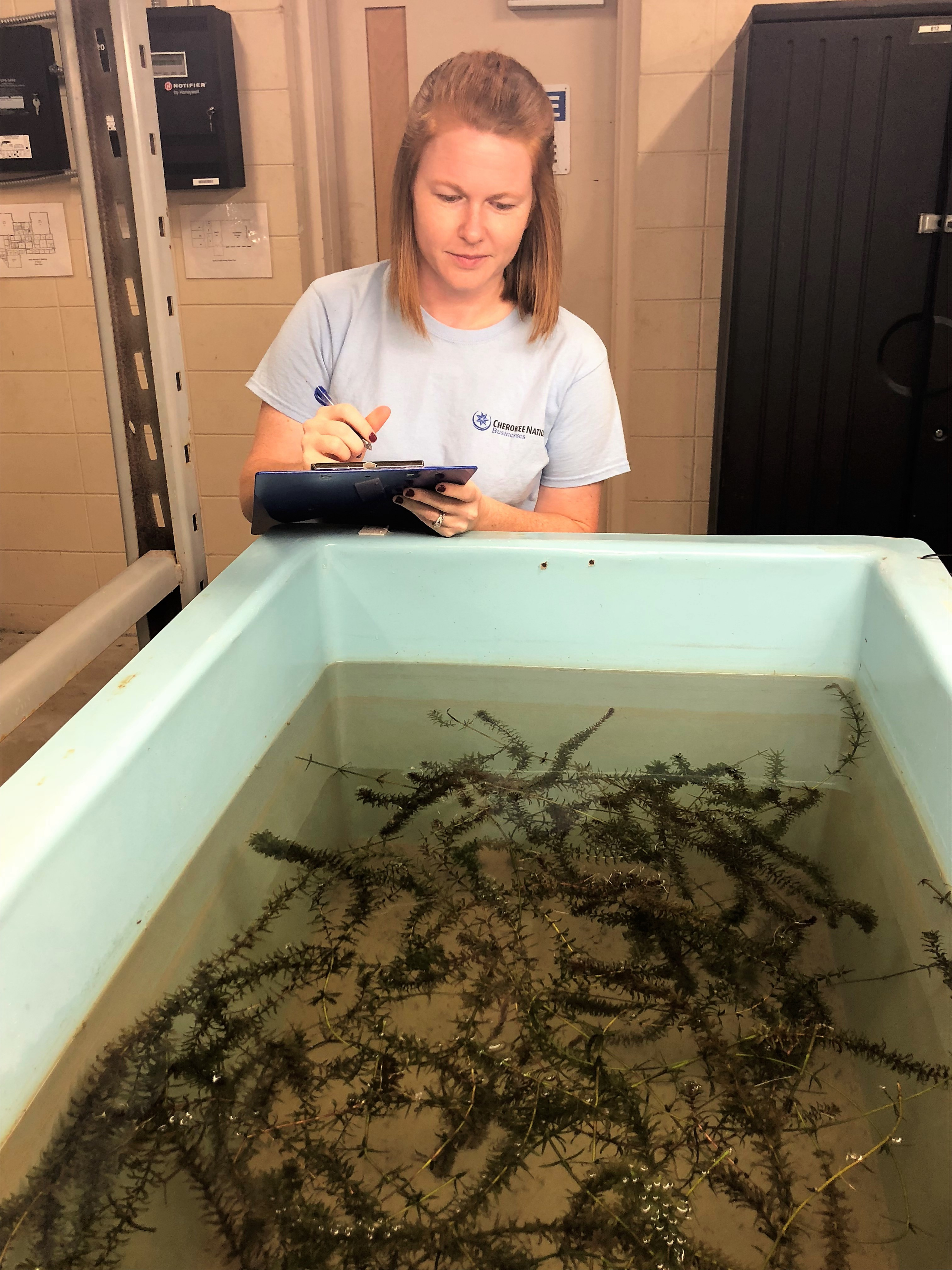}
     \includegraphics[scale=0.45]{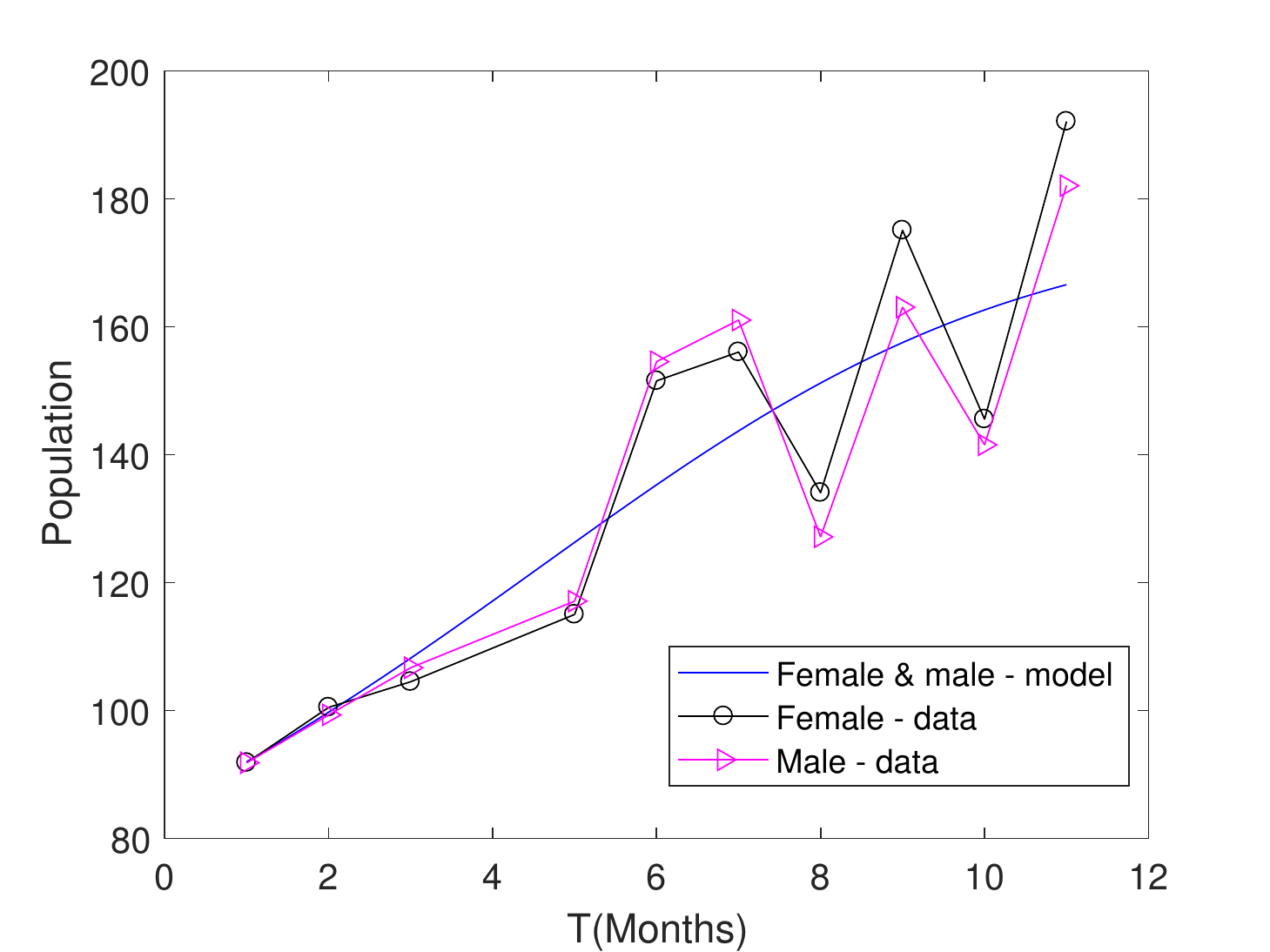}
  }
 \caption{\small Left panel: This is a tank from the USGS facility where the mesocosm experiment was run with fancy guppies (\textit{Poecilia reticulata}). Right panel: preliminary population data over 11-month period. The smooth curve is a result of simulating the basic population model \eqref{TYC_1} (with $s=0$) with best-fit parameters $\beta=0.0057, ~\delta=0.0648,$ and $K=405$.
 \label{Fig:fig2n1} 
 }
 \end{figure}

\section{Equilibrium and Stability Analysis}
\subsection{Equilibrium and Stability Analysis when $\boldsymbol{\mu>0}$}

We relegate a significant amount of the local stability analysis to the appendix see section \ref{app}. We begin by stating the following theorem,

\begin{theorem}
\label{thm_2}
Let $\mu>0$. The boundary equilibrium, $(0,0,\frac{\mu}{\delta})$, of system \eqref{TYC_1} is locally stable if $\frac{1}{9}<\delta< 1$ and unstable if $0 < \delta \leq \frac{1}{9}$.
\end{theorem}
\begin{proof}
Now consider the boundary equilibrium of model \eqref{TYC_1}, which satisfies $f^*=0$, is $(0,0,\frac{\mu}{\delta})$. Now evaluate \eqref{Jac_1} at $(0,0,\frac{\mu}{\delta})$. Then the characteristic equation about $(0,0,\frac{\mu}{\delta})$ is 
\beq
\label{jj_18}
\lambda^3+3\delta \lambda^2+3\delta^2 \lambda+\delta^2=0,
\eeq
with corresponding eigenvalues
\beq
\label{jj_19}
\begin{split}
& \lambda_1=(\delta^3-\delta^2)^\frac{1}{3}-\delta,\\
& \lambda_2=-\frac{1}{2}(\delta^3-\delta^2)^\frac{1}{3}-\delta+\frac{\sqrt{3}}{2}(\delta^3-\delta^2)^\frac{1}{3} i,\\
& \lambda_3=-\frac{1}{2}(\delta^3-\delta^2)^\frac{1}{3}-\delta-\frac{\sqrt{3}}{2}(\delta^3-\delta^2)^\frac{1}{3} i.
\end{split}
\eeq
Since $0<\delta<1$, then $\lambda_1<0$. The inequality $-\frac{1}{2}(\delta^3-\delta^2)^\frac{1}{3}-\delta<0$ implies that $\delta>\frac{1}{9}$. According to the Routh Hurwitz stability criteria, the system \eqref{TYC_1} under $f^*=0$ is locally stable if $\delta>\frac{1}{9}$.
\end{proof}

\subsection{Equilibrium and Stability Analysis when $\boldsymbol{\mu=0}$}

If $\mu=0$, then $s^*=0$. By direct computations, we can verify that the equilibrium of model \eqref{TYC_1} is of the form 
\beq
\label{20}
(f^*, m^*, 0), \quad \mathrm{with} \quad  f^*=m^*.
\eeq
Let $f^*_{\pm}=m^*_{\pm}=\frac{K}{4} \pm \frac{K}{4} \sqrt{1-\frac{16\delta}{\beta K}}$. If $16\delta=\beta K$, then the model \eqref{TYC_1} has two equilibria for which $f^*=\frac{K}{4}$ and $0$, respectively. If $16\delta< \beta K$, the model has 3 different equilibria with $f^*= f^*_+,~ f^*_-$ and $0$. If $16\delta>\beta K$, the model has only one trivial equilibrium, $f^*=0$. 

In the case $16\delta=\beta K$, the equilibria are $(\frac{K}{4}, \frac{K}{4}, 0)$ and $(0,0,0)$, repectively. The corresponding characteristic equation about $(\frac{K}{4}, \frac{K}{4}, 0)$ is 
\beq
\label{jj_21}
\lambda^3+(3\delta-\frac{K\beta}{16})\lambda^2+(3\delta^2-\frac{K\beta \delta}{8})\lambda-\frac{\delta^2(K\beta -16\delta)}{16}=0,
\eeq
with eigenvalues 
\beq
\label{jj_22}
\begin{split}
& \lambda_1=-\delta<0, \\
& \lambda_2=-\delta<0, \\
& \lambda_3=\frac{1}{16}K\beta-\delta=0. 
\end{split}
\eeq
By Routh Hurwitz criterion, $(\frac{K}{4}, \frac{K}{4}, 0)$ is unstable. Similarly, we find $(0,0,0)$ is locally stable. 

In the case $16\delta<\beta K$, the equilibria are $(f^*_+, m^*_+,0), (f^*_-, m^*_-,0)$ and $(0,0,0)$. The characteristic equation about $(f^*_+, m^*_+,0)$ is 
\beq
\label{jj_23}
(\lambda+\delta)^2\left(\lambda+\frac{1}{8}K \beta \sqrt{\frac{K\beta-16\delta}{k\beta}}+\frac{1}{8}K\beta-2\delta\right)=0
\eeq
The corresponding eigenvalues are $\lambda_1=\lambda_2=-\delta<0$, and $\lambda_3=-\frac{1}{8}K \beta \sqrt{\frac{K\beta-16\delta}{k\beta}}-\frac{1}{8}K\beta+2\delta$. Since $16\delta<\beta K$ then $\lambda_3<0$. This implies $(f^*_+, m^*_+,0)$ is locally stable. 

Similarly, we find the characteristic equation about $(f^*_-, m^*_-,0)$ is 
\beq
\label{jj_24}
(\lambda+\delta)^2\left[\lambda-\left(\frac{1}{8}K \beta \sqrt{\frac{K\beta-16\delta}{K\beta}}-\frac{1}{8}K\beta+2\delta\right)\right]=0.
\eeq
It is easy to verify the third eigenvalue $\left(\frac{1}{8}K \beta \sqrt{\frac{K\beta-16\delta}{K\beta}}-\frac{1}{8}K\beta+2\delta \right)>0$ under $16\delta<\beta K$. Therefore $ (f^*_-, m^*_-,0)$ is unstable. 

In the case of  $16\delta>\beta K$, the only equilibrium is $(0,0,0)$, which is locally stable.


\begin{remark}
The three and four species TYC models are now known to blow-up in finite time \cite{PBT19}. Thus one must exhibit caution when dealing with such models by restricting initial conditions, as well as the rate of feminized male/feminized super male introduction $\mu$. In the current manuscript we restrict initial data and the size of $\mu$, s.t. we always have positive bounded solutions, and all of the ensuing optimal control theory can be applied. 
\end{remark}

%
%
%
%
%
%
%
%
%
%
%
%
%
%


\subsection{Optimal Control Analysis}
The goal of this section is to investigate the mechanisms in our TYC system of equations, that, if controlled, could lead to optimal levels of both wild type female and male densities. We assume that the introduction rate $\mu$ is not known \emph{a priori} and enter the system as a time-dependent control. The response for the range is $0 \leq \mu(t) < \infty$. 

Consider the following objective function
\beq
J_0(\mu)=\int^{T}_{0} -(f+m)- \frac{1}{2}\mu^2 dt,
\eeq
subject to the governing equations and initial conditions. Optimal strategies are derived for the objective function, where we want to minimize both female and male populations and also minimizing the YY males introduction rate $\mu$. Optimal controls are searched for within the set $U_0$, namely,
\beq
 U_0= \{ \mu~|~\mu \ \mbox{measurable}, \  0 \leq \mu < \infty,  \  t \in [0,T], \ \forall T\}.
\eeq
The goal is to seek an optimal $\mu^{*}(t)$ such that,
\beq
\begin{split}
J_0(\mu^*) &=\underset{\mu} {\max} \int^{T}_{0} -(f+m)- \frac{1}{2}\mu^2 dt \\
&= \underset{\mu}{\min} \int^{T}_{0} f+m+\frac{1}{2}\mu^{2}  dt.
\end{split}
\eeq

Consider the following existence theorem,

\begin{theorem}
\label{thm_tyc_1_oct}
Consider the optimal control problem \eqref{TYC_1} with $\mu=\mu(t)$. There exists $\mu^{*}(t) \in U_0$ such that
\begin{equation}
  J_0(\mu^{*}) = \underset{\mu}{\max} \int^{T}_{0} -(f+m) - \frac{1}{2}\mu^{2}  dt\\
\end{equation}
\end{theorem}

\begin{proof}
The compactness (closed and bounded in the ODE case) of the functional $J$ follows from the global boundedness of the state variables and the control $\mu$. Also the functional $J$ is concave in the argument $\mu$. This is easily verified via standard application \cite{FR75}. These facts in conjunction give the existence of an optimal control.
\end{proof}

We use Pontryagin's maximum principle to derive the necessary conditions on the optimal control. The Hamiltonian for $J_0$ is given by
\beq
\label{TYC_1_OCT_ham}
H_0=-(f+m) - \frac{1}{2}\mu^2 +\lambda_1 f'+\lambda_2 m'+\lambda_3 s'.
\eeq
We use the Hamiltonian to find a differential equation of the adjoint $\lambda_i, i=1,2,3.$ Namely,
\beq
\begin{split}
& \lambda_1'(t)= 1-\lambda_1 [ \frac{m \beta}{2}(1-\frac{f+m+s}{K})-\frac{f  m \beta}{2K}-\delta]\\
&\quad \quad \quad -\lambda_2[(\frac{m}{2}+s)\beta(1-\frac{f+m+s}{K})-\frac{\beta}{K}(\frac{fm}{2}+fs)],\\
&\lambda_2'(t)=1-\lambda_1 [ \frac{f \beta}{2}(1-\frac{f+m+s}{K})-\frac{f  m \beta}{2K}]\\
& \quad \quad \quad -\lambda_2 [\frac{f \beta}{2}(1-\frac{f+m+s}{K})-\frac{f m \beta}{2K}-\delta],\\
&\lambda_3'(t)=\frac{\lambda_1 f m \beta}{2K} +\lambda_2[\frac{\beta}{K}(\frac{fm}{2}+fs)-\beta f (1-\frac{f+m+s}{K})]+\lambda_3 \delta,
\end{split}
\eeq
with the transversality condition given as 
\beq
\lambda_1(T)=\lambda_2(T)=\lambda_3(T)=0.
\eeq
Now considering the optimality conditions, the Hamiltonian function is differentiated with respect to control variable $\mu$ resulting in
\beq
\frac{\partial H_0}{\partial \mu}=\lambda_3-\mu.
\eeq
Then a compact way of writing the optimal control $\mu$ is 
\begin{equation}
\label{TYC_1_opt_1}
\mu^{*}(t)=\max(0, \lambda_3).
\end{equation}
The following theorem encapsulates the above.
\begin{theorem}
An optimal control $\mu^* \in U_0$ for the system \eqref{TYC_1} that maximizes the objective functional $J$ is characterized by \eqref{TYC_1_opt_1}.
\end{theorem}


\section{Mirroring TYC startegy via Harvesting}
\label{PHM}
A key issue in the implementation of the TYC strategy is the ``production'' of Trojan YY males. 

\begin{remark}
Here we present a model that can be reared effectively in the laboratory, while mirroring the TYC strategy - but that does not require YY males/females.  
Essentially the TYC strategy works by reducing the number of females in each generation, whilst increasing the number of males. We implement this in the current model where we remove a specified fraction of females whilst adding in males.  This strategy is called {\em female harvesting male stocking} (FHMS) and attempts to mirror the manipulation of the sex-ratio via the TYC strategy. We will introduce several forms of harvesting/stocking. 
\end{remark}

Consider the following model (We also note this model as Model 1)
\begin{equation}
\label{TYC_3}
\begin{split}
& \frac{df}{dt}=\frac{1}{2}fm\beta L-\delta f-\eta_1 f,\\
& \frac{dm}{dt}=\frac{1}{2}fm\beta L-\delta m+\eta_2 m,
\end{split}
\end{equation}
with $L=1-\frac{f+m}{K}.$ Here, $\eta_1$ and $\eta_2$ are non-negative parameters the specify the removal rate of females and addition rate of males, respectively. It is assumed that $\delta>\eta_2$ if $\eta_2 \not = 0$.  

\subsection{Equilibria and Stability Analysis}

An equilibrium of $f^*=0$ of this system requires $m^*=0$.  It is easy to verify that the trivial equilibrium $(0,0)$ is locally stable since both eigenvalues of the linearized systems are negative.

Let $f^*>0$ and consider the three cases:

\begin{enumerate}
\item $\eta_1>0$ and $\eta_2=0$,
\item $\eta_1=0$ and $\eta_2>0$,
\item $\eta_1>0$ and $\eta_2>0$.
\end{enumerate}

\subsubsection{Case 1: $\eta_1>0, \eta_2=0$.} In this situation the equilibrium of \eqref{TYC_3} is 
\beq
m^*=\left(1+\frac{\eta_1}{\delta}\right)f^*
\eeq
where $f^*$ satisfies
\beq
\frac{1}{2}\left(2+\frac{\eta_1}{\delta}\right)\beta f^2-\frac{1}{2}\beta K f+K\delta=0.
\eeq
Define
\beq
\Delta=(\frac{1}{2}\beta K)^2-2(2+\frac{\eta_1}{\delta})\beta K\delta.
\eeq
If $\Delta=0$, that is, $\beta K=8(2\delta+\eta_1)$, then $f^*=\frac{K}{2(2+\frac{\eta_1}{\delta})}=\frac{4\delta}{\beta} $
and $m^*=\frac{4(\delta+\eta_1)}{\beta}$. In such a case, the corresponding characteristic equation is 
\begin{equation}
\label{chara_3_1}
\lambda^2+\frac{2\delta(\delta+\eta_1)\lambda}{2\delta+\eta_1}=0.
\end{equation}
Since $0$ is a root of \eqref{chara_3_1} then by Routh Hurwitz Criterion, $\left(\frac{4\delta}{\beta},\frac{4(\delta+\eta_1)}{\beta}\right)$ is not stable. 

If $\Delta>0$, that is, $\beta K>8(2\delta+\eta_1)$, then we have the equilibrium solutions $(f^*_+, m^*_+)$ and $(f^*_-, m^*_-)$ where
\beqn
f^*_{\pm}&=&\frac{ (\beta K \pm \sqrt{\beta K (\beta K-16\delta-8\eta_1} )\delta}{2(2\delta+\eta_1)\beta},\\
m^*_{\pm}&=&\frac{(\delta+\eta_1)(\beta K \pm \sqrt{\beta K(\beta K-16\delta-8\eta_1)})}{2(2\delta+\eta_1)\beta}.
\eeqn
Notice that both $f^*_+$ and $f^*_-$ are positive. The characteristic equation about $(f^*_+,m^*_+)$ is given by 
\begin{equation}
\label{equli_4_1}
\lambda^2+k_{41}\lambda+k_{40}=0,
\end{equation}
where 
\beqn
\label{k_41}
k_{41}&=&\delta\frac{(\delta+\eta_1)\sqrt{K\beta(K\beta-16\delta-8\eta_1)}}{2(2\delta+\eta_1)^2}+\delta \frac{K\beta \delta+K\beta \eta_1-8\delta^2-12\delta \eta_1-4\eta_1^2}{2(2\delta+\eta_1)^2}, \\
\label{k_40}
k_{40}&=&\delta \frac{(\delta+\eta_1)\sqrt{K\beta(K\beta-16\delta-8\eta_1)}}{4(2\delta+\eta_1)}+\delta \frac{K\beta \delta+K\beta \eta_1-16 \delta^2-24 \delta \eta_1-8\eta_1^2}{4(2\delta+\eta_1)}.
\eeqn
By Routh Hurwitz criterion, the system with characteristic equation \eqref{equli_4_1} is locally stable if and only if both coefficients satisfy $k_{4i}>0$ for $i=0,1$. Since $\beta K>8(2\delta+\eta_1)$, then $\sqrt{K\beta(K\beta-16\delta-8\eta_1)}>0$.  Likewise,
\beqn
& & K\beta \delta+K\beta \eta_1-8 \delta^2-12\delta \eta_1-4\eta_1^2 \\
& &\quad > 8 \delta(2\delta+\eta_1)+8\eta_1 (2\delta+\eta_1)-8\delta^2-12\delta \eta_1-4\eta_1^2\\
& &\quad = 16\delta^2+8\delta \eta_1+16\eta_1 \delta+8\eta_1^2-8\delta^2-12\delta \eta_1-4\eta_1^2\\
& &\quad = 8\delta^2+12\delta \eta_1+4\eta_1^2 > 0.
\eeqn
Similarly,
\beqn
&& K\beta \delta+K\beta \eta_1-16\delta^2-24\delta \eta_1-8\eta_1^2 \\
&&\quad > 8 \delta(2\delta+\eta_1)+8\eta_1 (2\delta+\eta_1)-16\delta^2-24\delta \eta_1-8\eta_1^2\\
&&\quad =16\delta^2+8\delta \eta_1+16\eta_1 \delta+8\eta_1^2-16\delta^2-24\delta \eta_1-8\eta_1^2 =0.
\eeqn
This implies that $k_{41}>0$ and $k_{40}>0$.  Subsequently, the equilibrium $(f^*_+,m^*_+)$ is locally stable. Now consider the characteristic equation about $(f^*_-,m^*_-)$, namely,
\beq
\label{chara_5_1}
\lambda^2+k_{51}\lambda+k_{50}=0.
\eeq
with
\beqn
k_{51}&=&\delta\frac{K\beta \delta+K\beta \eta_1-8\delta^2-12\delta \eta_1-4\eta_1^2-(\delta+\eta_1)\sqrt{K\beta(K\beta-16\delta-8\eta_1)}}{2(2\delta+\eta_1)^2},\\
k_{50}&=&\delta \frac{K\beta \delta+K\beta \eta_1-16 \delta^2-24 \delta \eta_1-8\eta_1^2-(\delta+\eta_1)\sqrt{K\beta(K\beta-16\delta-8\eta_1)}}{4(2\delta+\eta_1)}.
\eeqn
The signs of $k_{51}$ and $k_{50}$ need to be known to quantify the stability of the equilibrium solution. Firstly, consider the numerator of $k_{50}$,
\beqn
&& K\beta \delta+K\beta \eta_1-16 \delta^2-24 \delta \eta_1-8\eta_1^2-(\delta+\eta_1)\sqrt{K\beta(K\beta-16\delta-8\eta_1)}\\
&&\quad =(\delta+\eta_1)(K\beta-(16\delta+8\eta_1))-(\delta+\eta_1)\sqrt{K\beta(K\beta-16\delta-8\eta_1)}.
\eeqn
To compare $(\delta+\eta_1)(K\beta-(16\delta+8\eta_1))$ and $(\delta+\eta_1)\sqrt{K\beta(K\beta-16\delta-8\eta_1)}$, it is enough to compare $(K\beta-(16\delta+8\eta_1))^2$ and $(\sqrt{K\beta(K\beta-16\delta-8\eta_1)})^2$.  Namely,
\beq
\begin{split}
& (K\beta-(16\delta+8\eta_1))^2-(\sqrt{K\beta(K\beta-16\delta-8\eta_1)})^2\\
&\quad =(K\beta)^2-2K\beta(16\delta+8\eta_1)+(16\delta+8\eta_1)^2-(K\beta)^2+K\beta(16\delta+8\eta_1)\\
&\quad =(16\delta+8\eta_1)^2-K\beta(16\delta+8\eta_1)\\
&\quad =(16\delta+8\eta_1)(16\delta+8\eta_1-K\beta)<0.
\end{split}
\eeq
This is less than zero because $\beta K>8(2\delta+\eta_1).$ Therefore, $k_{50}<0$ and $(f^*_-,m^*_-)$ is unstable. 

Lastly, if $\Delta<0$, that is, $\beta K< 8(2\delta+\eta_1)$, then the only equilibrium is trivial which is locally stable.

\subsubsection{Case 2: $\eta_1=0, \eta_2>0$.}  In the case of $\eta_1=0, \eta_2> 0$, the model simplifies to
\begin{equation}
\label{TYC_3_2}
\begin{split}
& \frac{df}{dt}=\frac{1}{2}fm\beta L-\delta f,\\
& \frac{dm}{dt}=\frac{1}{2}fm\beta L-\delta m+\eta_2 m,
\end{split}
\end{equation}
which is symmetric to the case $\eta_1>0, \eta_2=0$ if we replace $\eta_2$ by $-\eta_1$. Therefore, the following corollary can be developed. 
\begin{corollary} The boundary equilibrium $(0,0)$ is locally stable. There are 3 cases for the interior equilibria, 
\begin{enumerate}[label=(\roman*)]
\item  If $\beta K= 8(2\delta-\eta_2)$, the equilibrium of \eqref{TYC_3_2} is $\left(\frac{4(\delta-\eta_2)}{\beta}, \frac{4\delta}{\beta}\right)$ and it is locally stable.

\item If $\beta K>8(2\delta-\eta_2)$, the equilibria of \eqref{TYC_3_2} are\\

 \quad \quad $\left(\frac{(\delta-\eta_2)(\beta K \pm \sqrt{\beta K(\beta K-16\delta+8\eta_2)})}{2(2\delta+\eta_1)\beta}, \frac{ (\beta K \pm \sqrt{\beta K (\beta K-16\delta+8\eta_2} )\delta}{2(2\delta-\eta_2)\beta}\right)$ which the postive and negative branches are locally stable and unstable, respectively.
 %
%

\item If $\beta K<8(2\delta-\eta_2)$, the only equilibrium of \eqref{TYC_3_2} is $(0,0)$.

\end{enumerate}
 
\end{corollary}

\subsubsection{Case 3: $\eta_1>0, \eta_2>0$.} In the case of $\eta_1>0, \eta_2>0$, the equilibrium of model \eqref{TYC_3} is given by
\beq
m^*=\frac{\delta+\eta_1}{\delta-\eta_2}f^*,
\eeq
where $f^*$ satisfies 
\beq
(-2\beta \delta-\beta \eta_1+\beta \eta_2) f^2+(K\beta\delta-K\beta \eta_2) f-2K\delta^2+4K\delta \eta_2-2K\eta_2^2=0.
\eeq

If $\beta K=8 (2\delta+\eta_1-\eta_2)$, then $f^*=\frac{K\beta(\delta-\eta_2)}{2\beta(2\delta+\eta_1-\eta_2)}$ and $m^*=\frac{K\beta(\delta+\eta_1)}{2\beta(2\delta+\eta_1-\eta_2)}$. The corresponding characteristic equation is 
\beq
\lambda\left(\lambda+\frac{2(\delta-\eta_2)(\delta+\eta_1)}{2\delta+\eta_1-\eta_2}\right)=0.
\eeq
It is easy to see $(\frac{K\beta(\delta-\eta_2)}{2\beta(2\delta+\eta_1-\eta_2)},\frac{K\beta(\delta+\eta_1)}{2\beta(2\delta+\eta_1-\eta_2)})$ is unstable because one of the eigenvalues is $0$. 

If $\beta K>8 (2\delta+\eta_1-\eta_2)$, then we have two positive equilibria, $(f^*_+, m^*_+$ and $(f^*_-, m^*_-)$, where
\beq
\begin{split}
& f^*_ \pm=\frac{(K\beta \pm \sqrt{K \beta(K\beta-16 \delta-8\eta_1+8\eta_2)})(\delta-\eta_2)}{2\beta(2\delta+\eta_1-\eta_2)},\\
& m^*_\pm=\frac{(K\beta \pm \sqrt{K \beta(K\beta-16 \delta-8\eta_1+8\eta_2)})(\delta+\eta_1)}{2\beta(2\delta+\eta_1-\eta_2)}.
\end{split}
\eeq
The characteristic equation about $(f^*_+,m^*_+)$ is 
\beq
\label{chara_6}
\lambda^2+k_{61}\lambda+k_{60}=0
\eeq
with
\beq
\begin{split}
k_{61}=& \frac{-8\delta^3-12\delta^2 \eta_1+12\delta^2 \eta_2-4\delta\eta_1^2+16\delta \eta_1 \eta_2-4\delta \eta_2^2+4\eta_1^2\eta_2-4\eta_1 \eta_2^2}{2(2\delta+\eta_1-\eta_2)^2}\\
&+ \frac{\sqrt{K\beta (K\beta-16\delta-8\eta_1+8\eta_2)}(\delta^2+\delta \eta_1-\delta \eta_2-\eta_1 \eta_2)}{2(2\delta+\eta_1-\eta_2)^2}\\
&+\frac{K\beta(\delta^2+\delta \eta_1-\delta \eta_2-\eta_1 \eta_2)}{2(2\delta+\eta_1-\eta_2)^2}
\end{split}
\eeq
and
\beq
\begin{split}
k_{60}=& \frac{-16\delta^3-24\delta^2 \eta_1+24\delta^2 \eta_2-8\delta\eta_1^2+32\delta \eta_1 \eta_2-8\delta \eta_2^2+8\eta_1^2\eta_2-8\eta_1 \eta_2^2}{4(2\delta+\eta_1-\eta_2)}\\
&+ \frac{\sqrt{K\beta (K\beta-16\delta-8\eta_1+8\eta_2)}(\delta^2+\delta \eta_1-\delta \eta_2-\eta_1 \eta_2)}{4(2\delta+\eta_1-\eta_2)}\\
&+\frac{K\beta(\delta^2+\delta \eta_1-\delta \eta_2-\eta_1 \eta_2)}{4(2\delta+\eta_1-\eta_2)}.
\end{split}
\eeq
Since $\beta K>8 (2\delta+\eta_1-\eta_2)$ and $\delta>\eta_2$, then 
\beq
\begin{split}
&\sqrt{K\beta (K\beta-16\delta-8\eta_1+8\eta_2)}(\delta^2+\delta \eta_1-\delta \eta_2-\eta_1 \eta_2)\\
&\quad =\sqrt{K\beta (K\beta-16\delta-8\eta_1+8\eta_2)}(\delta+\eta_1)(\delta-\eta_2) > 0.
\end{split}
\eeq
Likewise, the remaining two terms in the numerator of $k_{61}$ can be combined and  shown to be greater than zero, namely,
\beq
\begin{split}
&-8(\delta+\eta_1)(\delta-\eta_2)(\delta+\frac{1}{2}\eta_1-\frac{1}{2}\eta_2)+K\beta(\delta+\eta_1)(\delta-\eta_2)\\
&\quad = (\delta+\eta_1)(\delta-\eta_2)(K\beta-8(\delta+\frac{1}{2}\eta_1-\frac{1}{2}\eta_2))\\
&\quad > (\delta+\eta_1)(\delta-\eta_2)(K\beta-16(\delta+\frac{1}{2}\eta_1-\frac{1}{2}\eta_2))\\
&\quad =(\delta+\eta_1)(\delta-\eta_2)(K\beta-8 (2\delta+\eta_1-\eta_2)) > 0.
\end{split}
\eeq
Therefore $k_{61}>0$.  In addition, since
\beq
\begin{split}
&-16\delta^3-24\delta^2 \eta_1+24\delta^2 \eta_2-8\delta\eta_1^2+32\delta \eta_1 \eta_2-8\delta \eta_2^2+8\eta_1^2\eta_2-8\eta_1 \eta_2^2\\
&+K\beta(\delta^2+\delta \eta_1-\delta \eta_2-\eta_1 \eta_2)\\
&=(\delta+\eta_1)(\delta-\eta_2)(\beta K-8 (2\delta+\eta_1-\eta_2)) >0,
\end{split}
\eeq
then $k_{60}>0.$ Therefore, $(f^*_+,m^*_+)$ is locally stable. 

As for $(f^*_-,m^*_-)$, the corresponding characteristic equation is 
\beq
\label{chara_7}
\lambda^2+k_{71}\lambda+k_{70}=0
\eeq
where 
\beq
\begin{split}
k_{71}=& \frac{-8\delta^3-12\delta^2 \eta_1+12\delta^2 \eta_2-4\delta\eta_1^2+16\delta \eta_1 \eta_2-4\delta \eta_2^2+4\eta_1^2\eta_2-4\eta_1 \eta_2^2}{2(2\delta+\eta_1-\eta_2)^2}\\
&+ \frac{\sqrt{K\beta (K\beta-16\delta-8\eta_1+8\eta_2)}(-\delta^2-\delta \eta_1+\delta \eta_2+\eta_1 \eta_2)}{2(2\delta+\eta_1-\eta_2)^2}\\
&+\frac{K\beta(\delta^2+\delta \eta_1-\delta \eta_2-\eta_1 \eta_2)}{2(2\delta+\eta_1-\eta_2)^2}
\end{split}
\eeq
and
\beq
\begin{split}
k_{70}=& \frac{-16\delta^3-24\delta^2 \eta_1+24\delta^2 \eta_2-8\delta\eta_1^2+32\delta \eta_1 \eta_2-8\delta \eta_2^2+8\eta_1^2\eta_2-8\eta_1 \eta_2^2}{4(2\delta+\eta_1-\eta_2)}\\
&+ \frac{\sqrt{K\beta (K\beta-16\delta-8\eta_1+8\eta_2)}(-\delta^2-\delta \eta_1+\delta \eta_2+\eta_1 \eta_2)}{4(2\delta+\eta_1-\eta_2)}\\
&+\frac{K\beta(\delta^2+\delta \eta_1-\delta \eta_2-\eta_1 \eta_2)}{4(2\delta+\eta_1-\eta_2)}.
\end{split}
\eeq
The denominator of $k_{70}$ and $k_{71}$ are positive.  The numerator of $k_{70}$ can be simplified to
\begin{equation}
\label{sign_7}
\begin{split}
(\delta+\eta_1)(\delta-\eta_2)(K\beta-16\delta-8 \eta_1+8 \eta_2-\sqrt{K\beta (K\beta-16\delta-8\eta_1+8\eta_2)}).
\end{split}
\end{equation}
The sign of \eqref{sign_7} depends on $(K\beta-16\delta-8 \eta_1+8 \eta_2)-\sqrt{K\beta (K\beta-16\delta-8\eta_1+8\eta_2)}$.  It is enough to compare $(K\beta-16\delta-8 \eta_1+8 \eta_2)^2$ and $(\sqrt{K\beta (K\beta-16\delta-8\eta_1+8\eta_2)})^2$ because both of them are positive. Therefore, 
\begin{equation*}
\begin{split}
& (K\beta-16\delta-8 \eta_1+8 \eta_2)^2 - (\sqrt{K\beta (K\beta-16\delta-8\eta_1+8\eta_2)})^2 \\
&\quad = (K\beta-16\delta-8 \eta_1+8 \eta_2)(K\beta-16\delta-8 \eta_1+8 \eta_2-K\beta)\\
&\quad = (K\beta-16\delta-8 \eta_1+8 \eta_2)(-16\delta-8 \eta_1+8 \eta_2)\\
&\quad = -8((\delta+\eta_1)+(\delta-\eta_2))(K\beta-16\delta-8 \eta_1+8 \eta_2) <0.
\end{split}
\end{equation*}
This shows that $k_{70}<0$.  Subsequently, $(f^*_-,m^*_-)$ is unstable by Routh Hurwitz criterion. 

The local stability for the equilibria in each case is now understood. The global stability of the trivial equilibrium $(0,0)$ is stated in the following theorem.
\begin{theorem}
\label{thm_global_3}
Consider the model \eqref{TYC_3}, the equilibrium $(0,0)$ is globally asymptotically stable if $\beta K<2\delta+\eta_1- \eta_2.$
\end{theorem}
\begin{proof}
Consider the Lyapunov function $V = f+m$. It is left to show that $\frac{dV}{dt}<0$ for all $(f,m) \not =(0,0)$.
Consider 
\beq
\begin{split}
\frac{dV}{dt}&=\frac{df}{dt}+\frac{dm}{dt}\\
&=fm\beta L-(\delta+\eta_1)f-(\delta-\eta_2)m\\
&\leq fm\beta-(\delta+\eta_1)f-(\delta-\eta_2)m, \quad \mathrm{since} \quad L \leq 1 \\
&\leq [\beta-\frac{1}{m}(\delta+\eta_1)-\frac{1}{f}(\delta-\eta_2)] fm\\
&\leq [\beta-\frac{1}{K}(\delta+\eta_1)-\frac{1}{K}(\delta-\eta_2)] fm.
\end{split}
\eeq
It is enough to show $\beta-\frac{1}{K}(\delta+\eta_1)-\frac{1}{K}(\delta-\eta_2)<0$. By direct calculation, we have $\beta K<2\delta+\eta_1-\eta_2$, which proves theorem \eqref{thm_global_3}.
\end{proof}

Various other forms of harvesting for FHMS are also tried, but we maintain the same general structure that is females are harvested via a term $-\eta_1 G_1(f)$, whilst males are added to the system via a term $\eta_2 G_2(m)$. The results spanning various forms of $G_{1}, G_{2}$ are relegated to the appendix, section \ref{AppendixB}.

\subsection{Optimal Control Analysis}
The goal of this section is to further investigate controls the female-male system. In particular, assume that the removal rate $\eta_{1}$ or the addition rate $\eta_{2}$ are not known \emph{a priori} and enter the system as time-dependent controls. The responses are over the range $0 \leq \eta_1, \eta_2 \leq 1$. For clarity, the model is restated with the temporal dependence in $\eta_i$,
\begin{equation}
\label{TYC_3_OCT}
\begin{split}
& \frac{df}{dt}=\frac{1}{2}fm\beta L-\delta f-\eta_1(t) f,\\
& \frac{dm}{dt}=\frac{1}{2}fm\beta L-\delta m+\eta_2(t) m.
\end{split}
\end{equation}

Consider the objective function
\beq
  J_1(\eta_{1},\eta_{2})= \int^{T}_{0} -(f+m)- \frac{1}{2}(\eta_{1}^{2} +\eta_{2}^{2} ) dt.
\eeq
subject to the governing equations in \eqref{TYC_3_OCT} and specified initial conditions. Optimal strategies are derived for the following objective function, where we minimize both female and male populations while minimizing the harvesting and addition rates. Optimal controls are searched for within the set $U_1$, namely,
\beq
 U_1= \{ (\eta_{1},\eta_{2})|\eta_{i} \ \mbox{measurable}, \  0 \leq \eta_{1} \leq 1, 0 \leq \eta_{2} < \delta, \  t \in [0,T], \ \forall T\}.
\eeq
The goal is to seek an optimal $(\eta_{1}^{*}, \eta_{2}^{*})$ such that
\beq
  J_1(\eta_{1}^{*}, \eta_{2}^{*})= \underset{(\eta_{1},\eta_{2})}{\max} \int^{T}_{0} (-(f+m) - \frac{1}{2}(\eta_1^2+\eta_{2}^{2})  dt
\eeq
The following existence theorem is stated.

\begin{theorem}
\label{thm_tyc_3_oct}
Consider the optimal control problem \eqref{TYC_3_OCT}. There exists $(\eta_{1}^{*}, \eta_{2}^{*}) \in U_1$ such that
\beq
\begin{split}
  J_1(\eta_{1}^{*}, \eta_{2}^{*}) &= \underset{(\eta_{1},\eta_{2})}{\max} \int^{T}_{0} -(f+m) - \frac{1}{2}(\eta_1^2+\eta_{2}^{2})  dt\\
  & = \underset{(\eta_{1},\eta_{2})}{\min} \int^{T}_{0} f+m + \frac{1}{2}(\eta_1^2+\eta_{2}^{2})  dt.\\
  \end{split}
\eeq
\end{theorem}

\begin{proof}
Since $f+m \leq K,~ 0 \leq \eta_1 \leq 1$ and $0 \leq \eta_2 \leq \delta \leq 1$, then the compactness of the functional $J$ follows from the global boundedness of the state variables and the controls $\eta_1$ and $\eta_2$. Also the functional $J$ is concave in the both of its arguments $\eta_{1}$ and $\eta_{2}$. This is easily verified via standard application \cite{FR75}. Subsequently, these facts in conjunction provide the existence of an optimal control.
\end{proof}

We use the Pontryagin's maximum principle to derive the necessary conditions on the optimal controls. Consider the Hamiltonian for $J_1$, namely,
\beq
\label{OCT_ham}
H_1=-(f+m) - \frac{1}{2}(\eta_1^2+\eta_{2}^{2}) +\lambda_1 f'+\lambda_2 m'.
\eeq
The Hamiltonian is used to establish a differential equation of the adjoint $\lambda_i, i=1,2.$ That is,
\beq
\begin{split}
&\lambda_1'(t)=1-\lambda_1 [\frac{m \beta }{2} (1-\frac{f+m}{K})-\frac{f m \beta}{2K}-\delta-\eta_1]-\lambda_2 [\frac{ m \beta}{2} (1-\frac{f+m}{K})-\frac{ f m \beta}{2K}],\\
&\lambda_2'(t)=1-\lambda_1 [\frac{f \beta }{2} (1-\frac{f+m}{K})-\frac{f m \beta}{2K}]-\lambda_2 [\frac{ f \beta}{2} (1-\frac{f+m}{K})-\frac{ f m \beta}{2K}-\delta+\eta_2].
\end{split}
\eeq
with the transversality condition of $\lambda_1(T)=\lambda_2(T)=0.$ In consideration of the optimality conditions, the Hamiltonian is differentiated with respect to the control variables $\eta_1$ and $\eta_2$ resulting in
\beq
\begin{split}
& \frac{\partial H_1}{\partial \eta_1}=-f\lambda_1-\eta_1,\\
& \frac{\partial H_1}{\partial \eta_2}=m\lambda_2-\eta_2.
\end{split}
\eeq
We find a characterization of $\eta_1$ by considering three cases,

\begin{enumerate}
\item If $\frac{\partial H_1}{\partial \eta_{1}}<0$ then $\eta_{1}^{*}=0$.  This implies that $-f\lambda_1<0$;
\item If $\frac{\partial H_1}{\partial \eta_{1}}=0$ then $\eta_{1}^{*}=-f\lambda_1$. This implies that $0 \leq - f\lambda_1 \leq 1$;
\item If $\frac{\partial H_1}{\partial \eta_{1}}>0$ then $\eta_{1}^{*}=1$.  This implies that $1<- f\lambda_1$.
\end{enumerate}

Notice the recurrence of the expression $-f\lambda_1$.  This expression is strictly less than 0 when the control is at the lower bound.  In constrast, when the control is at the upper bound, then $-f\lambda_1>1$. Thus a compact way of writing the optimal control $\eta_1$ is 
\begin{equation}
\label{opt_1}
\eta_1(t)^{*}=\min(1, \max(0, -f\lambda_1)).
\end{equation}
Similarly, for $\eta_2$, 
\begin{equation}
\label{opt_2}x
\eta_2(t)^{*}=\min(1, \max(0, m \lambda_2)).
\end{equation}
\begin{theorem}
An optimal control $(\eta_1^*, \eta_2^*) \in U_1$ for the system \eqref{TYC_3} that maximizes the objective functional $J$ is characterized by \eqref{opt_1}-\eqref{opt_2}.
\end{theorem}

\begin{remark}
The optimal control anlysis of the nonlinear forms of FHMS model is relegated to the appendix, section \ref{AppendixB}.
\end{remark}

\section{Optimal Strategy Comparisons between Classic TYC Model and FHMS Model}
In this section numerical simulations are used to compare the optimal controls for the classic TYC model to the FHMS model.  The wild-type female and male populations are compared for each control. The Guppy fish population data from USGS is utilized to run the simulations with our best fit parameters, that is, $\beta=0.005774, ~\delta=.0.0648,$ and $K=405.0705.$

\begin{table}[H]
\caption{Parameters used for numerical simulations}
\begin{tabular}{|c|l|r|}
\hline
\textbf{Parameter }            & ~~\textbf{Description}       & \textbf{Value} \\ \hline
$ \beta $  & ~~Birth rate        &~~ 0.0057 ~~\\ \hline
$ \delta $ & ~~Death rate        &~~ 0.0648 ~~\\ \hline
$K$                     & ~~Carrying capacity & 405 ~~   \\ \hline
$T$                     & ~~Terminal time     & 200 ~~ \\ \hline
\end{tabular}
\end{table}

\begin{figure}[H]
        \begin{minipage}[b]{0.48\linewidth}
            \centering
            \includegraphics[width=\textwidth]{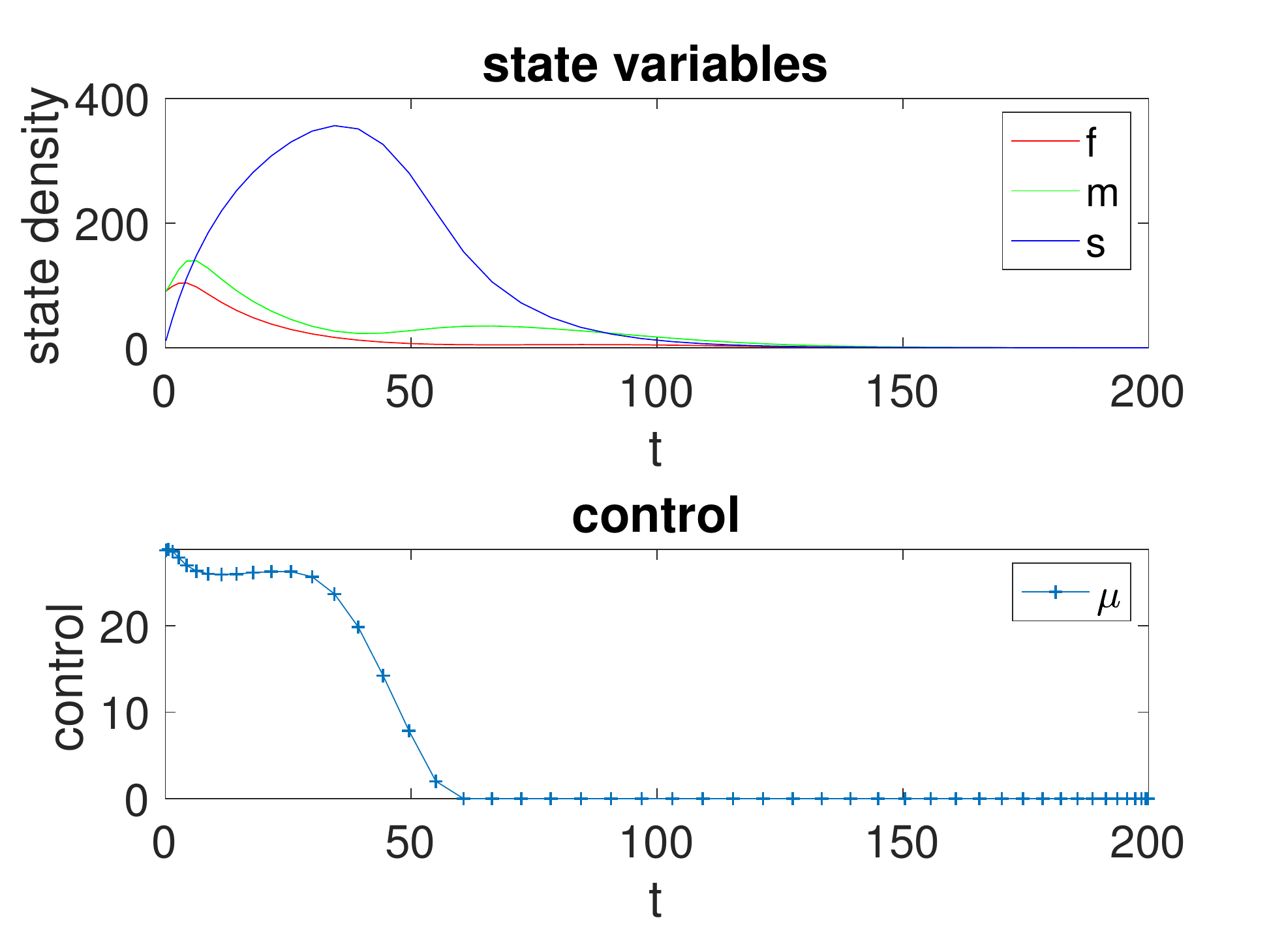}
            \subcaption{ }
             \label{fig:fig_1A}
        \end{minipage}
        \hspace{0.10cm}
        \begin{minipage}[b]{0.480\linewidth}
            \centering
            \includegraphics[width=\textwidth]{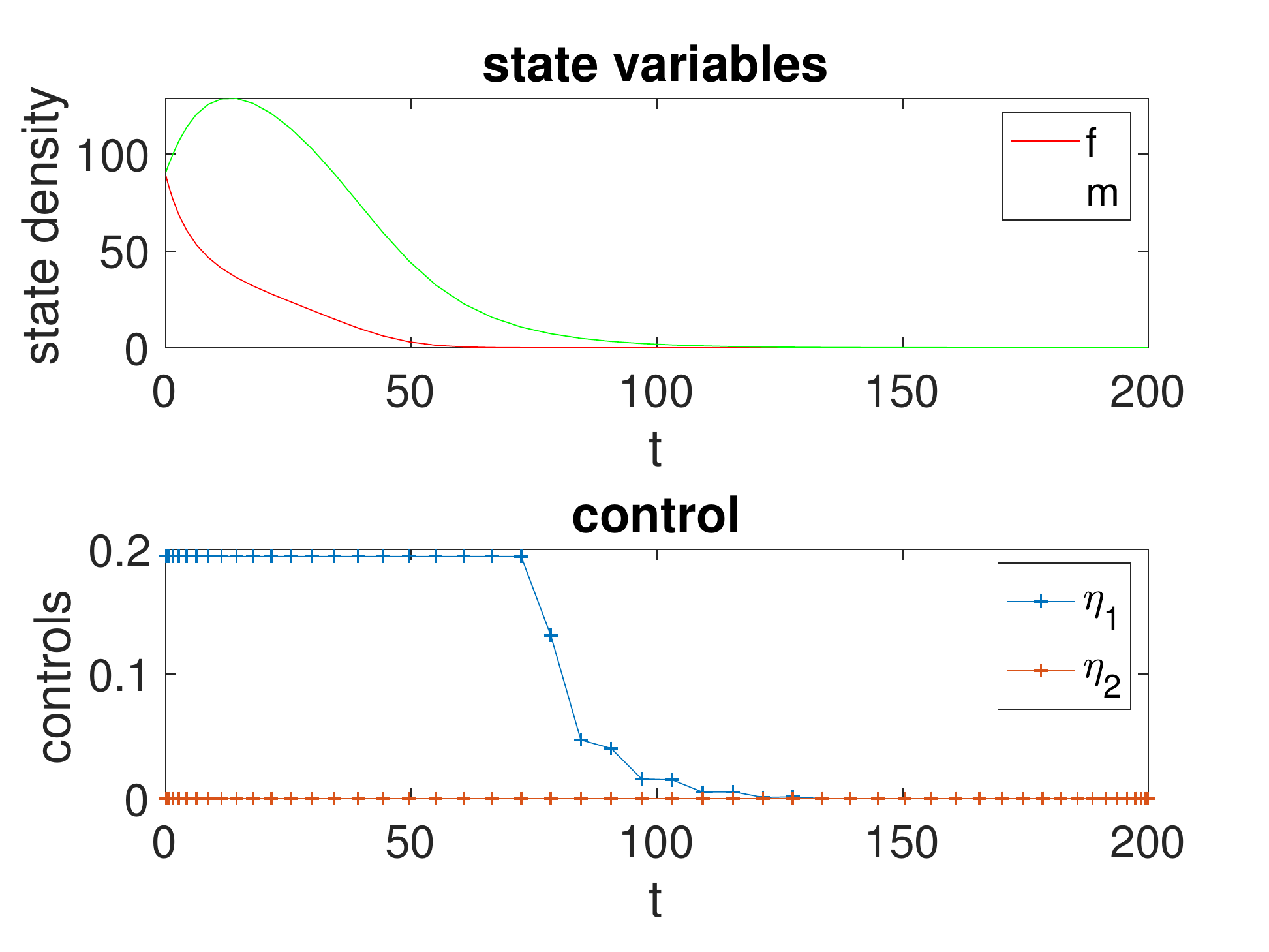}
            \subcaption{ }
             \label{fig:fig_1B}
        \end{minipage}
\caption{(A) Female (red), male (green) and supermale (blue) densities and optimal control $\mu(t)$ in \eqref{TYC_1} change with time $t$. (B) Female (red) and male (green) densities and optimal controls $\eta_{1}(t)$ and $\eta_{2}(t)$ in \eqref{TYC_3_OCT} change with time $t$.  }
\end{figure}

\begin{figure}[H]
        \begin{minipage}[b]{0.480\linewidth}
            \centering
            \includegraphics[width=\textwidth]{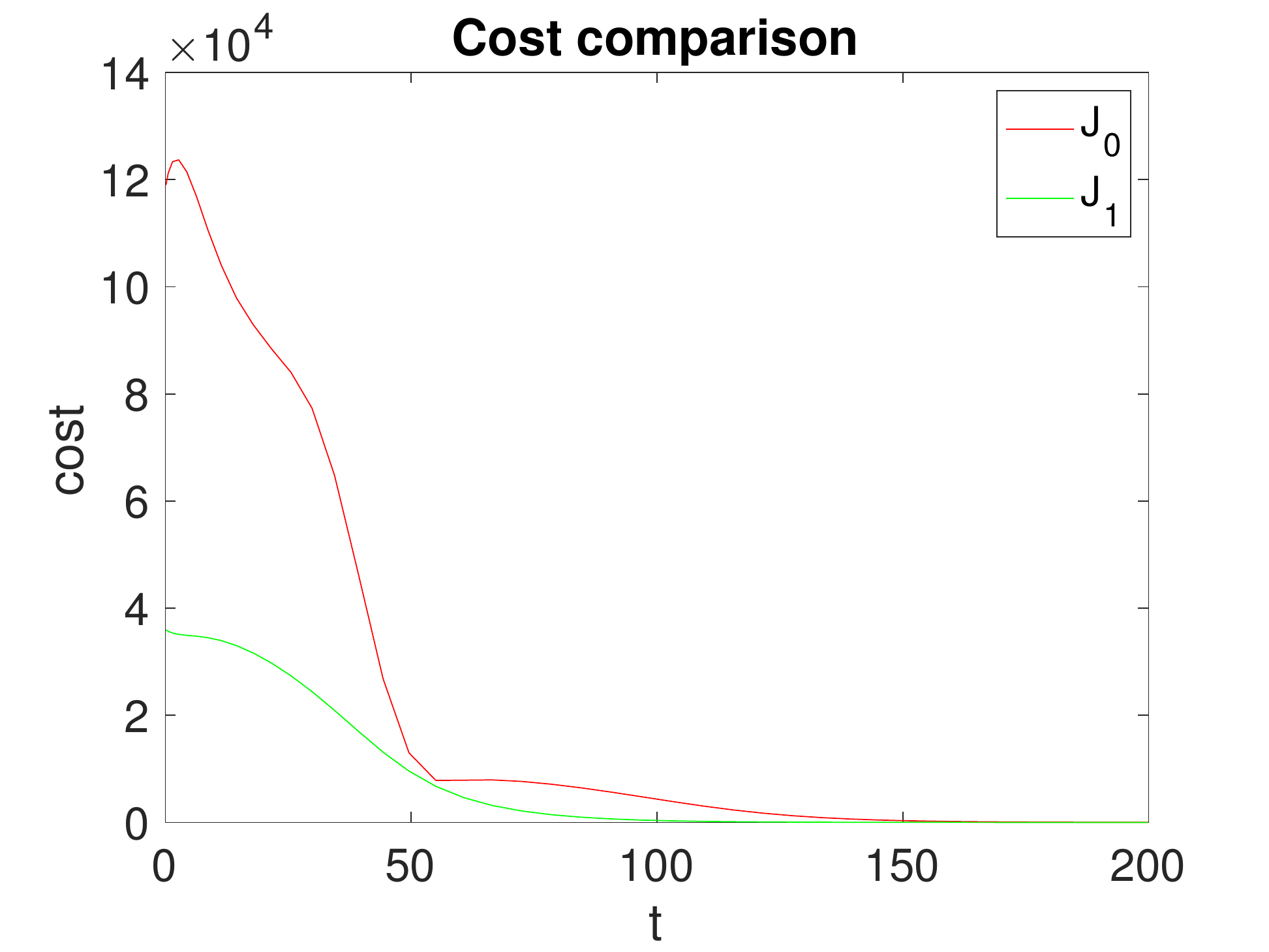}
            \subcaption{ }
             \label{fig:fig6_1}
        \end{minipage}
        \hspace{0.10cm}
        \begin{minipage}[b]{0.480\linewidth}
            \centering
            \includegraphics[width=\textwidth]{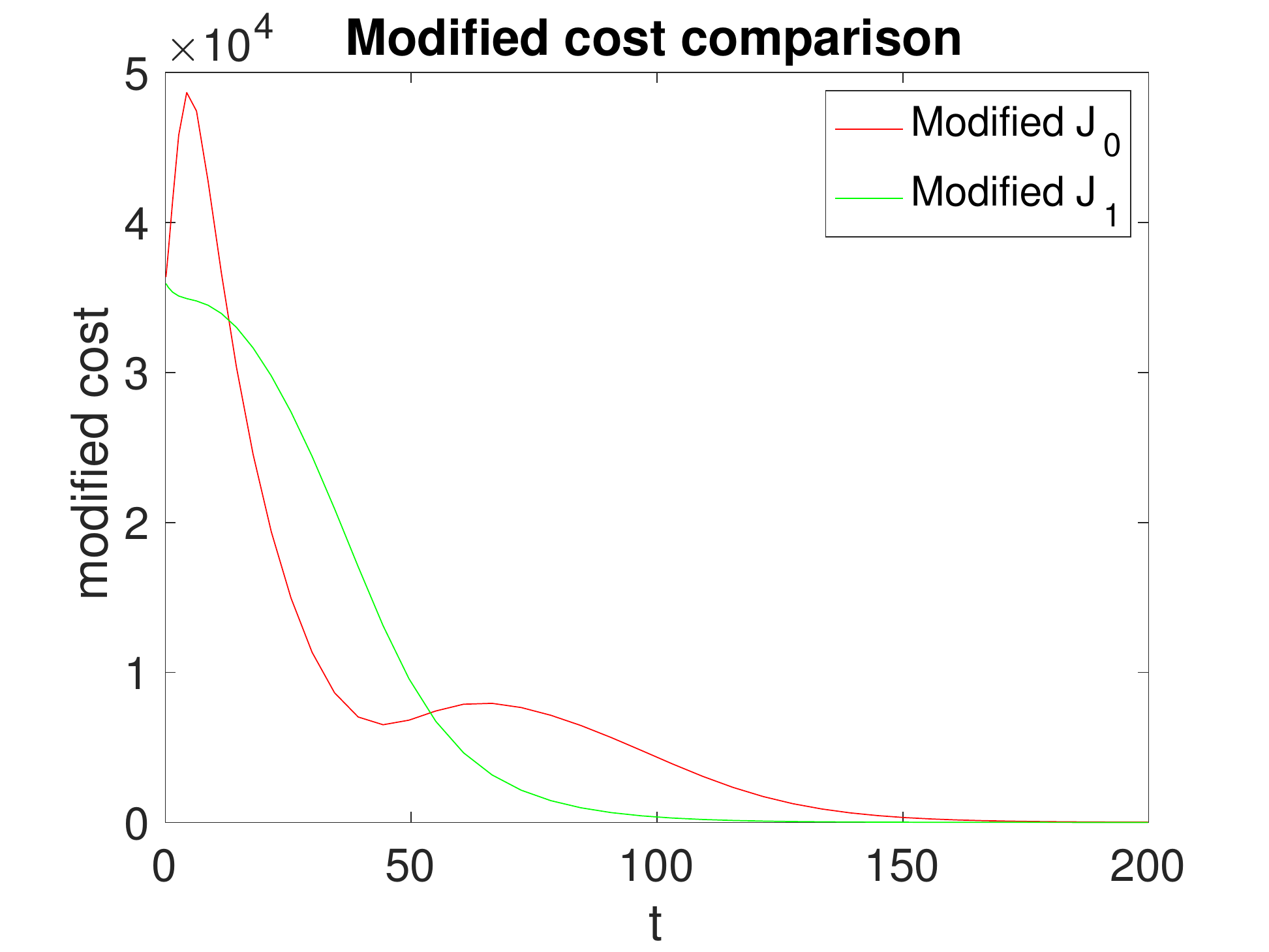}
            \subcaption{ }
             \label{fig:fig6_2}
        \end{minipage}
\caption{(A) In this simulation we look at the objective function towards the optimal implement for each model, we clearly see the cost function of TYC model (Model 0) is much larger than FHMS model (Model 1). (B) Since the constraints and scales for $\mu$ and $\eta_{1}, \eta_{2}$ are different, it's hard to compare the effectiveness for each strategy, so here we also look at the cost deducting the controls, that is, $\int^{T}_{0} f+m$  $dt $ under optimal control.}
\end{figure}

In the classical TYC case, supermales are introduced at high rate for longer time as shown in Fig.~\ref{fig:fig_1A}. Once the females are brought below a particular threshold, the introduction of supermales decreases and finally turns off as extinction nears. Similarly to the TYC strategy, the female removal rate $\eta_1$ keeps high for a certain time to reduce the female population and gradually declines as the entire population approaches extinction.

\begin{remark}
From Fig.~\ref{fig:fig_1B}, we see that introducing males to the entire population is not helpful for the eradication. Note, the optimal control $\eta_{2}$ turns out to be essentially 0.
\end{remark}

From Fig.~\ref{fig:fig6_1}, it is clear that the cost function of FHMS strategy at each time is much less than the classical TYC strategy. We also consider that the range for $\mu (0<\mu<\infty)$ and $\eta_1, \eta_2 (0 \leq \eta_1, \eta_2 \leq 1)$ are different, we calculate the the cost excluding the controls under the optimal control in each model as shown in the Fig.~\ref{fig:fig6_2}. And we get the following result:


\begin{table}[H]
\caption{Comparison of TYC and FHMS strategy results}
\begin{tabular}{|l|r|r|}
\hline
\quad\quad\textbf{Results}             & ~~\textbf{TYC}       & ~~\textbf{FHMS} \\ \hline
~Objective function  & 270820~~        & 84168  \\ \hline
~Cost deducting controls  & 106782~~        & 84156  \\ \hline
~Female population in final time  &  0.0184  & 0  \\ \hline
~Male population in final time &  0.005 & 0.0028  \\ \hline
~Female Approximate Eradication Time (month  )                  & 140 & 60   \\ \hline
~Male Approximate Eradication Time (month)                   & 156 & 110  \\ \hline
\end{tabular}
\end{table}

\begin{remark}
Note the approximate eradication time is defined as the time that the population is less than some $\epsilon$ where $0<\epsilon <1.$ Through this paper, we set $\epsilon=0.5$.
\end{remark}

\section{Alternative Harvesting Approach (FHMH)}
In this section we try an alternative harvesting approach.
\begin{remark}
When we compare the linear harvesting model to the classical TYC model, it is observed via Fig.~\ref{fig:fig_1B}, that the optimal levels $\eta_{2}$ of adding males in the harvesting model, is virtually zero. This is also seen in our various nonlinear models of harvesting, where we harvest females and introduce males, see section \ref{AppendixB}. Motivated by this observation, we analyze the optimal harvesting levels when both females and males are removed/harvested from the population. We call these \emph{female harvesting and male harvesting} models (FHMH).
\end{remark}

The general model is given by 
\begin{equation}
\label{NPH_0}
\begin{split}
& \frac{df}{dt}=\frac{1}{2}fm\beta L-\delta f-\eta_1 G_1(f),\\
& \frac{dm}{dt}=\frac{1}{2}fm\beta L-\delta m-\eta_2 G_2(m),
\end{split}
\end{equation}
where $G_1(f)$ and $G_2(m)$ are nonnegative functions. Again, assume that $\eta_1$ and $\eta_2$ are nonnegative. 

\subsection{Model 4: $\boldsymbol{G_1(f)=f,G_2(m)=m}$}
For clarity we restate the model:
\begin{equation}
\label{NPH_1}
\begin{split}
& \frac{df}{dt}=\frac{1}{2}fm\beta L-\delta f-\eta_1 f,\\
& \frac{dm}{dt}=\frac{1}{2}fm\beta L-\delta m-\eta_2 m.
\end{split}
\end{equation}
\subsubsection{Global stability}

The global stability of the trivial equilibrium $(0,0)$ is stated in the following theorem.

\begin{theorem}
\label{NPH1_global}
The trivial equilibrium is globally asymptotically stable in \eqref{NPH_1} if $\beta K<2\delta+\eta_1+\eta_2$.
\end{theorem}
\begin{proof}
Again, consider the Lyapunov function $V=f+m$. It is left to to show $\frac{dV}{dt}<0$ for all $(f,m) \not =(0,0)$. Taking the derivative yields,
\beq
\begin{split}
\frac{dV}{dt}&=\frac{df}{dt}+\frac{dm}{dt}\\
&=fm\beta L-(\delta+\eta_1)f-(\delta+\eta_2)m\\
&\leq fm\beta-(\delta+\eta_1)f-(\delta+\eta_2)m, \quad \mathrm{since} \quad L \leq 1 \\
&\leq [\beta-\frac{1}{m}(\delta+\eta_1)-\frac{1}{f}(\delta+\eta_2)] fm\\
&\leq [\beta-\frac{1}{K}(\delta+\eta_1)-\frac{1}{K}(\delta+\eta_2)] fm.
\end{split}
\eeq
It is enough to show $\beta-\frac{1}{K}(\delta+\eta_1)-\frac{1}{K}(\delta+\eta_2)<0$. By direct calculation, we have $\beta K<2\delta+\eta_1+\eta_2$, which completes the proof.
\end{proof}

\begin{remark}
If one considers the condition for global asymptotic stability via theorem \ref{NPH1_global} that is $\beta K<2\delta+\eta_1+\eta_2$, to 
the condition for global asymptotic stability via theorem \ref{thm_global_3} that is $\beta K<2\delta+\eta_1- \eta_2$, we see this is stronger, and the weaker condition via 
theorem \ref{NPH1_global} enables global asymptotic stability - or extinction. Thus there is more merit in harvesting both males and females. Although counter intuitive to TYC methodology, the reason is that when we introduce normal males, there is always a chance they will mate with females - producing more females.
\end{remark}

\subsubsection{Optimal control analysis}
Assume that the removal rates are not known \emph{a priori} and enter the system as time-dependent controls. Optimal controls are sought within the range $0 \leq \eta_1, \eta_2 \leq 1$. Consider the objective function,
\beq
  J_4(\eta_{1},\eta_{2})= \int^{T}_{0} -(f+m)- \frac{1}{2}(\eta_{1}^{2} +\eta_{2}^{2} ) dt
\eeq
subject to \eqref{NPH_1} and with initial conditions $f(t_{0})=f_{0}, m(t_{0})=m_0$.  Optimal controls are sought that minimize the populations and remain in the set $U_4$, namely,
\beq
 U_4= \{ (\eta_{1},\eta_{2})|\eta_{i} \ \mbox{measurable}, \  0 \leq \eta_{1} \leq 1, 0 \leq \eta_{2} \leq 1, \  t \in [0,T], \ \forall T\}.
\eeq
Optimal functions $(\eta_{1}^{*}, \eta_{2}^{*})$ are sought such that,
\begin{equation}
\begin{split}
  J_4(\eta_{1}^{*}, \eta_{2}^{*}) &= \underset{(\eta_{1},\eta_{2})}{\max} \int^{T}_{0} (-(f+m) - \frac{1}{2}(\eta_1^2+\eta_{2}^{2})  dt\\
  & = \underset{(\eta_{1},\eta_{2})}{\min} \int^{T}_{0} f+m + \frac{1}{2}(\eta_1^2+\eta_{2}^{2})  dt.\\
  \end{split}
\end{equation}
Consider the following existence theorem.

\begin{theorem}
\label{thm_NPH1_oct}
Consider the optimal control problem \eqref{NPH_1}. There exists $(\eta_{1}^{*}, \eta_{2}^{*}) \in U_4$ such that
\begin{equation}
  J_4(\eta_{1}^{*}, \eta_{2}^{*}) = \underset{(\eta_{1},\eta_{2})}{\max} \int^{T}_{0} -(f+m) - \frac{1}{2}(\eta_1^2+\eta_{2}^{2})  dt
\end{equation}
\end{theorem}

\begin{proof}
Since $f+m \leq K,~ 0 \leq \eta_1 \leq 1$ and $0 \leq \eta_2  \leq 1$, the compactness of the functional $J$ follows from the global boundedness of the state variables and the controls $\eta_1$ and $\eta_2$. Also the functional $J_4$ is concave in both arguments, $\eta_{1}$ and $\eta_{2}$. These facts in conjunction give the existence of an optimal control.
\end{proof}

Pontryagin's maximum principle is used to derive the necessary conditions on the optimal controls. Consider the Hamiltonian for $J_4$, namely,
\beq
\label{NPH1_ham}
H_4=-(f+m) - \frac{1}{2}(\eta_1^2+\eta_{2}^{2}) +\lambda_1 f'+\lambda_2 m'.
\eeq
The Hamiltonian is used to find a differential equation of the adjoint $\lambda_i, i=1,2.$ That is,
\beq
\begin{split}
&\lambda_1'(t)=1-\lambda_1 \left[\frac{m \beta }{2} (1-\frac{f+m}{K})-\frac{f m \beta}{2K}-\delta-\eta_1\right]-\lambda_2 \left[\frac{ m \beta}{2} (1-\frac{f+m}{K})-\frac{ f m \beta}{2K}\right],\\
&\lambda_2'(t)=1-\lambda_1 \left[\frac{f \beta }{2} (1-\frac{f+m}{K})-\frac{f m \beta}{2K}\right]-\lambda_2\left[\frac{ f \beta}{2} (1-\frac{f+m}{K})-\frac{ f m \beta}{2K}-\delta-\eta_2\right]
\end{split}
\eeq
with the transversality condition $\lambda_1(T)=\lambda_2(T)=0$.

The Hamiltonian function is differentiated with respect to the control variables $\eta_1$ and $\eta_2$ resulting in
\beq
\begin{split}
& \frac{\partial H_4}{\partial \eta_1}=-f\lambda_1-\eta_1,\\
& \frac{\partial H_4}{\partial \eta_2}=-m\lambda_2-\eta_2
\end{split}
\eeq
As shown previously, a compact way of writing the optimal control for $\eta_i$ is
%
%
\begin{eqnarray}
\label{NPH1_opt_1}
\eta_1(t)^{*}&=&\min(1, \max(0, -f\lambda_1)),\\
\label{NPH1_opt_2}
\eta_2(t)^{*}&=&\min(1, \max(0, -m \lambda_2)).
\end{eqnarray}

\begin{theorem}
An optimal control $(\eta_1^*, \eta_2^*) \in U_4$ for the system \eqref{NPH_1} that maximizes the objective functional $J$ is characterized by \eqref{NPH1_opt_1}-\eqref{NPH1_opt_2}.
\end{theorem}
Various other forms of harvesting for FHMH are also tried, but we maintain the same general structure that is females are harvested via a term $-\eta_1 G_1(f)$ and males are harvested to the system via a term $-\eta_2 G_2(m)$. The results spanning various forms of $G_{1}, G_{2}$ are relegated to the appendix, section \ref{AppendixC}.

\subsection{Numerical Simulations and Comparisons}
In this section, we will numerically simulate the optimal strategy and its corresponding optimal states for the models 4-6 - the FHMH class of models. For models 5,6 the reader is referred to section \ref{AppendixC}. The parameters for simulating are given by:

\begin{table}[H]
\caption{Parameters used for numerical simulations}
\begin{tabular}{|c|l|r|}
\hline
\textbf{Parameter }            & ~~\textbf{Description}       & \textbf{Value} \\ \hline
$ \beta $  & ~~Birth rate        &~~ 0.0057 ~~\\ \hline
$ \delta $ & ~~Death rate        &~~ 0.0648 ~~\\ \hline
$K$                     & ~~Carrying capacity & 405 ~~   \\ \hline
$T$                     & ~~Terminal time     & 200 ~~ \\ \hline
$d_1$                 &  ~~Parameter In $G_1$ & 1~~\\ \hline
$d_2$                 &  ~~Parameter In $G_2$ & 1~~\\ \hline
\end{tabular}
\end{table}

\begin{figure}[H]
        \begin{minipage}[b]{0.480\linewidth}
            \centering
            \includegraphics[width=\textwidth]{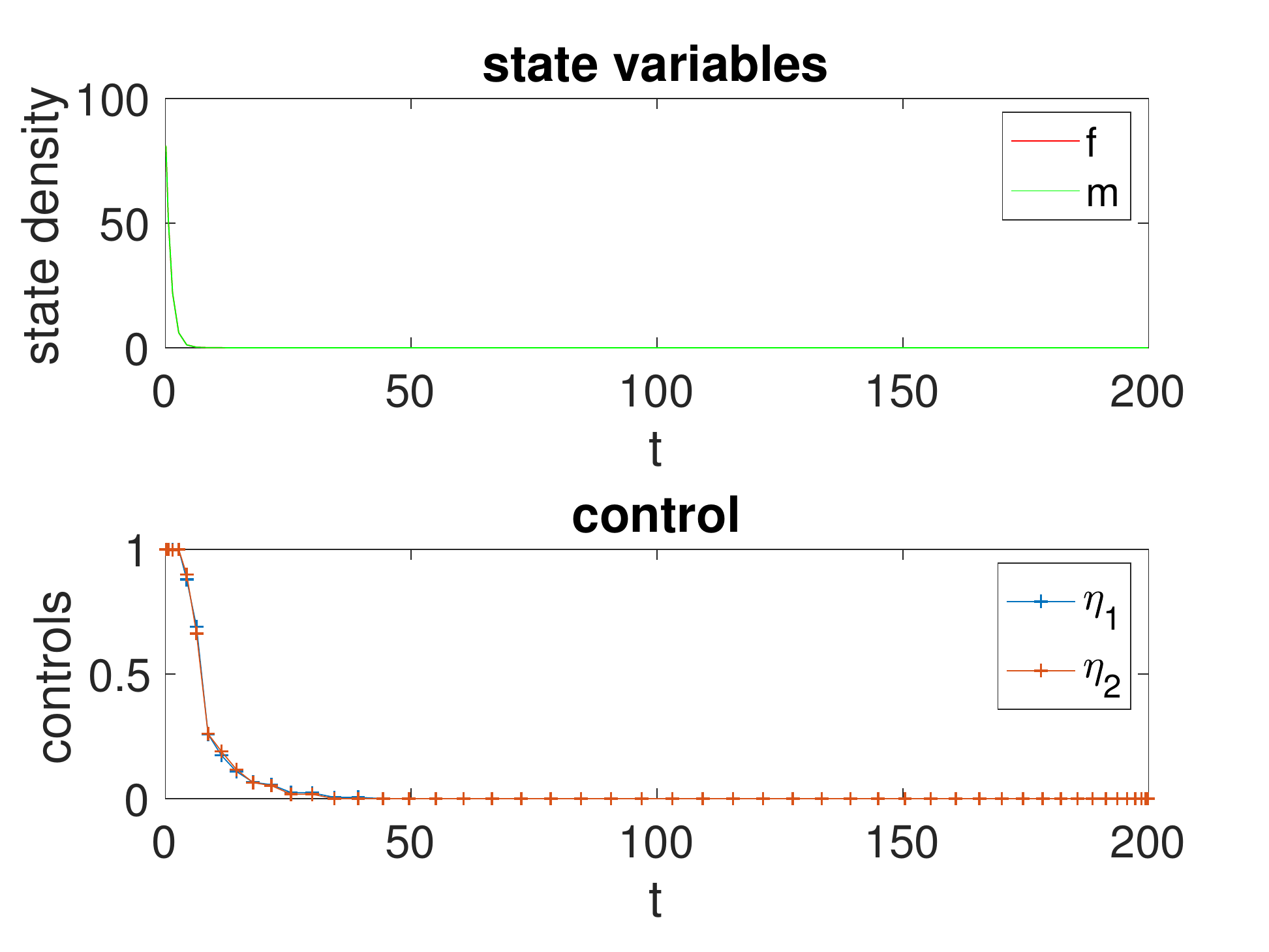}
            \subcaption{}
        \end{minipage}
        \hspace{0.10cm}
        \begin{minipage}[b]{0.480\linewidth}
            \centering
            \includegraphics[width=\textwidth]{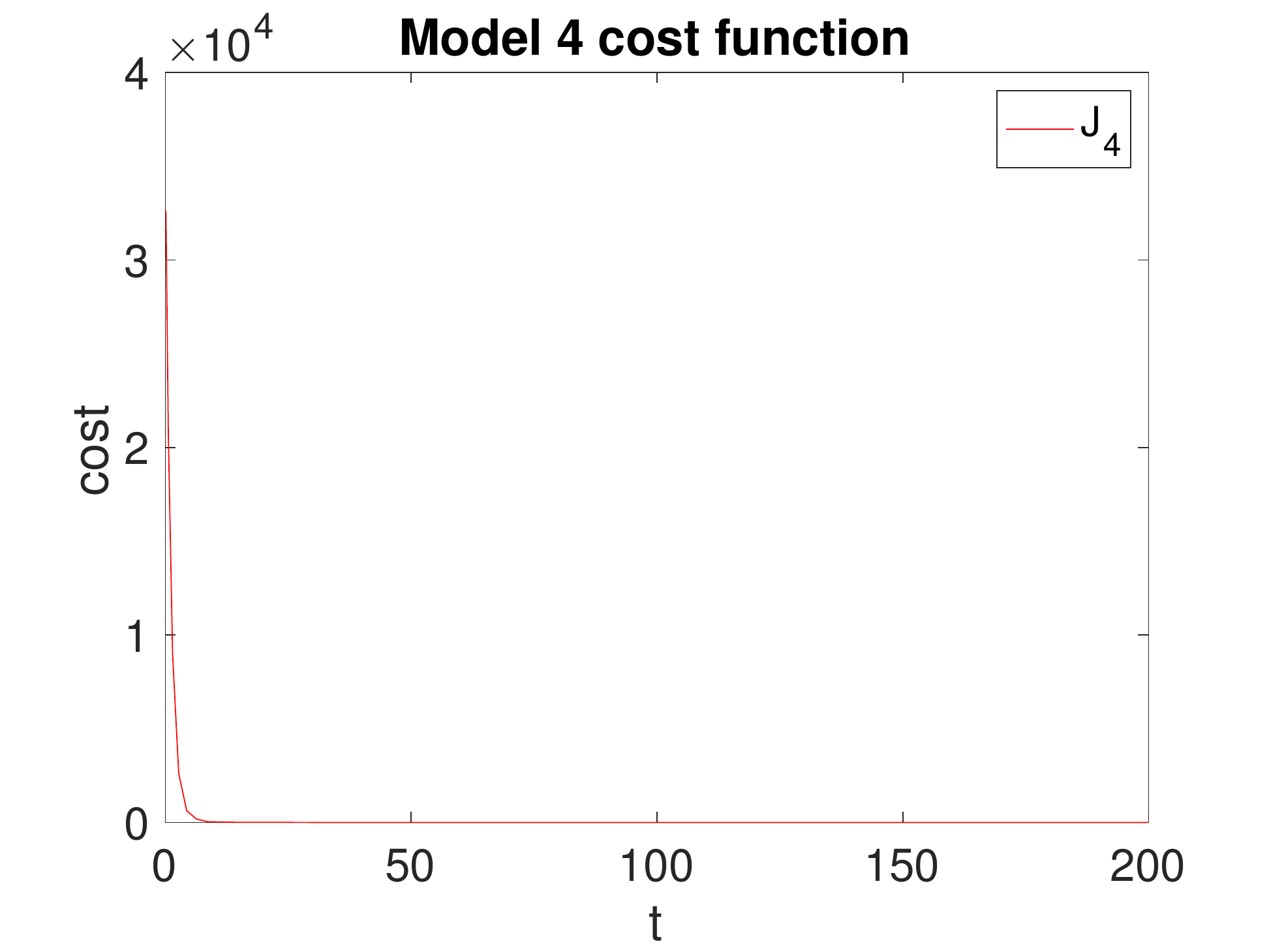}
            \subcaption{}
        \end{minipage}

\caption{(A) Female (red) and male (green) densities change with time $t$ with optimal control $\eta_1^*$ (blue) and $\eta_2^*$ (red) in Model 4. As we can see the controls and densities are indistinguishable. (B) Objective function $J_4$ decrease as increasing time with optimal control.}
        \label{model4}
\end{figure}

\begin{figure}[H]
        \begin{minipage}[b]{0.480\linewidth}
            \centering
            \includegraphics[width=\textwidth]{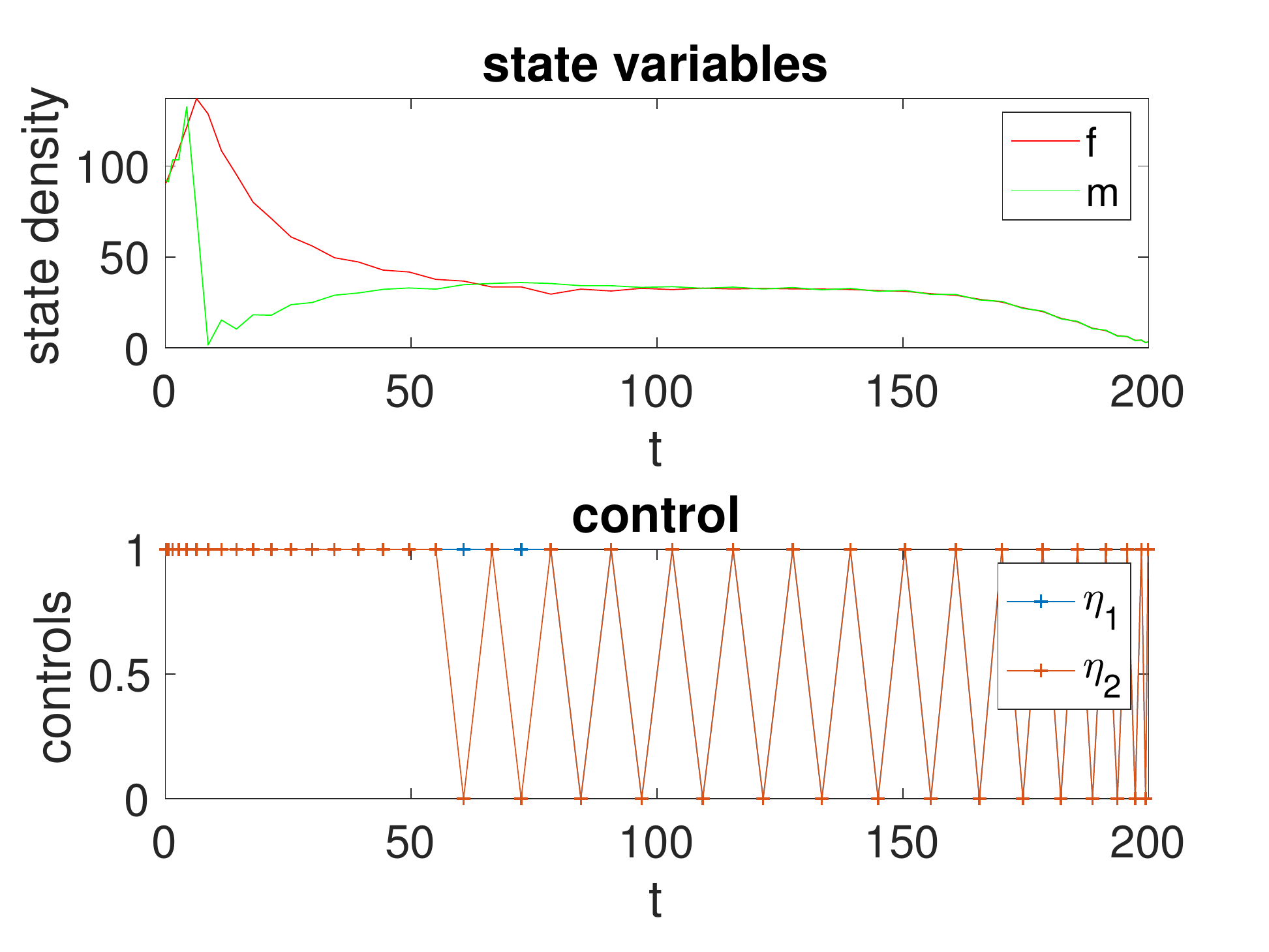}
            \subcaption{}
        \end{minipage}
        \hspace{0.10cm}
        \begin{minipage}[b]{0.480\linewidth}
            \centering
            \includegraphics[width=\textwidth]{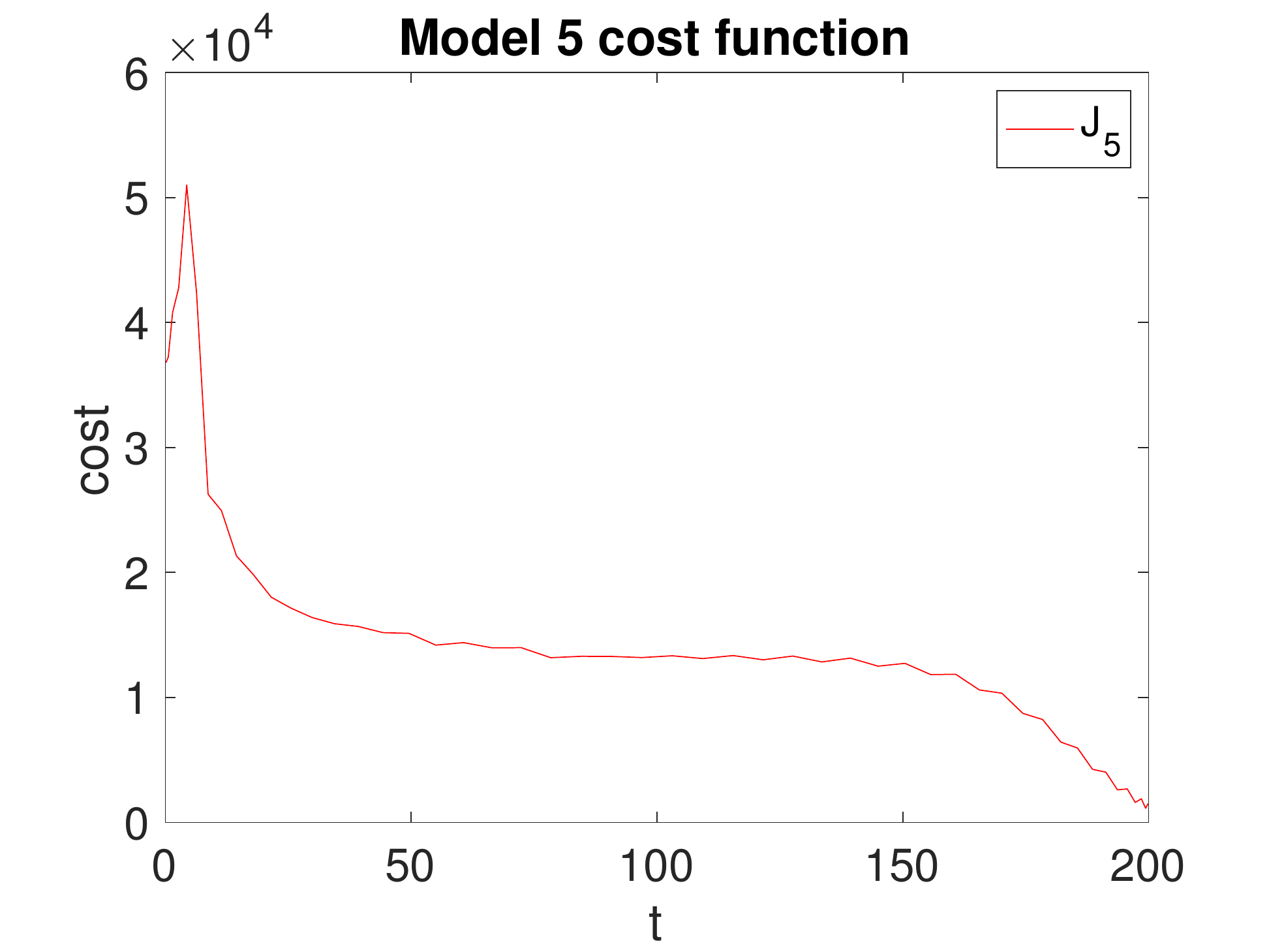}
            \subcaption{}
        \end{minipage}
\caption{(A) Female (red) and male (blue) densities change with time $t$ with optimal control $\eta_1^*$ (blue) and $\eta_2^*$ (red) in Model 5. (B) Objective function $J_5$ varies as increasing time with optimal control.}
      \label{model5}
\end{figure} 

\begin{figure}[H]
        \begin{minipage}[b]{0.480\linewidth}
            \centering
            \includegraphics[width=\textwidth]{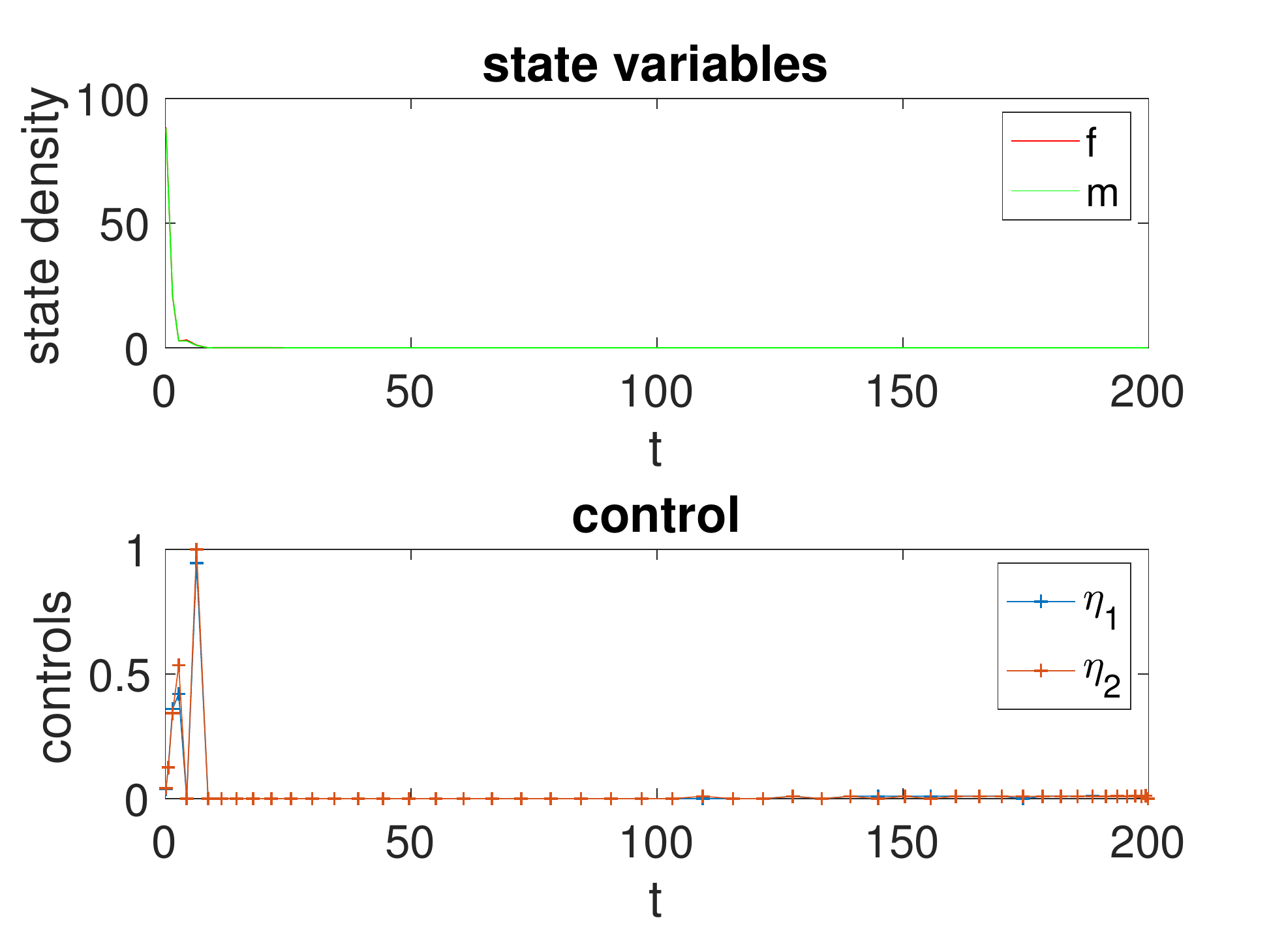}
            \subcaption{}
        \end{minipage}
        \hspace{0.10cm}
        \begin{minipage}[b]{0.480\linewidth}
            \centering
            \includegraphics[width=\textwidth]{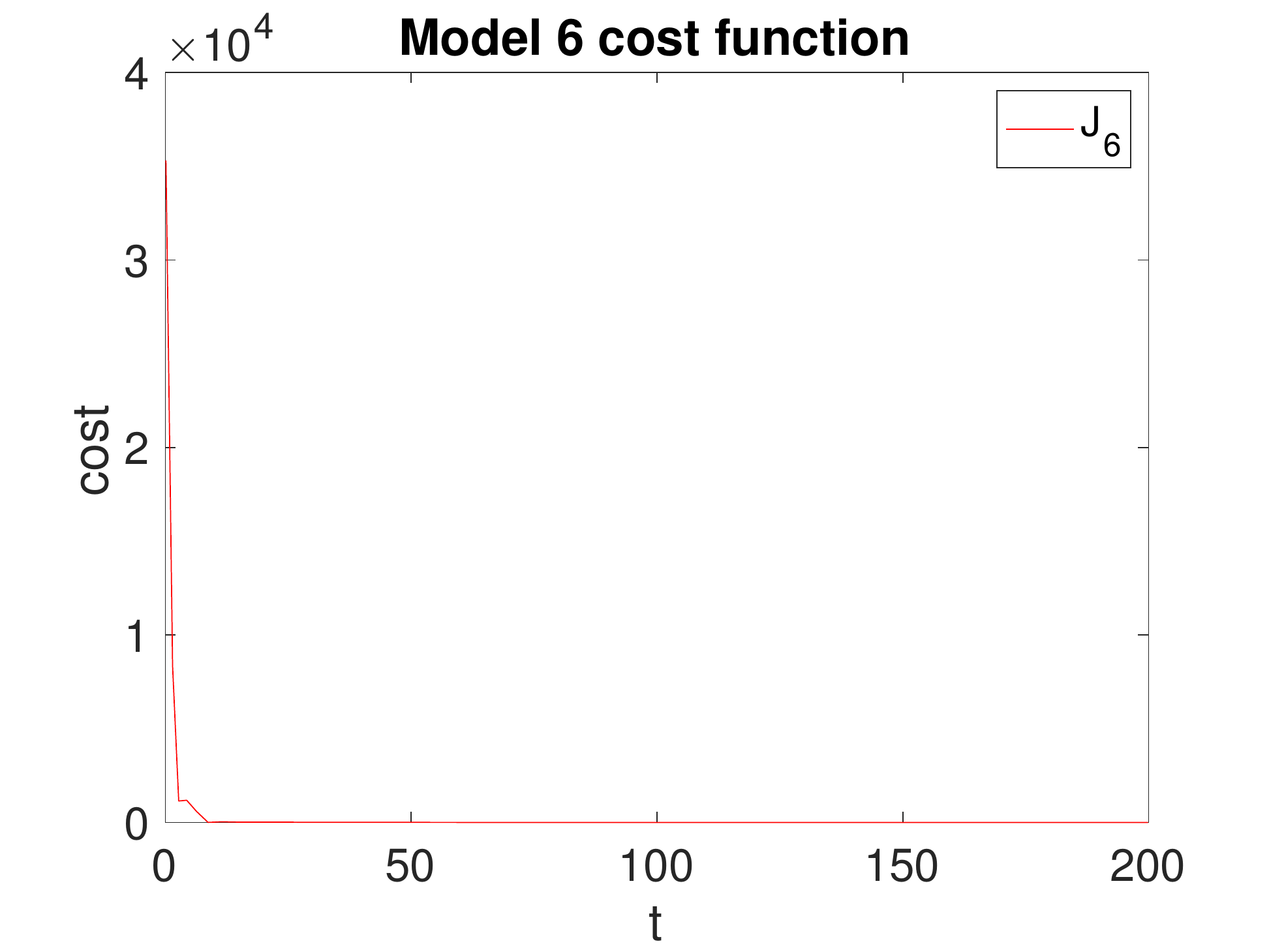}
            \subcaption{}
        \end{minipage}
\caption{(A) Female (red) and male (green) densities over time $t$ with optimal control $\eta_1^*$ (blue) and $\eta_2^*$ (red) in Model 6. We can see that the densities and controls are indistinguishable. (B) Objective function $J_6$ decrease as increasing time with optimal control.}
\label{model6}
\end{figure}


%
Table.~\ref{comp_all} compares all harvesting models including nonlinear forms of harvesting for FHMS (Model 2-3) and FHMH (Model 5-6), which are refered to section \ref{AppendixB} and \ref{AppendixC}, respectively. 
\begin{table}[H]
\caption{Comparisons among Model 1-Model 6}
\resizebox{\textwidth}{14mm}{
\begin{tabular}{|c|c|c|c|c|c|c|} 
\hline
\textbf{Results}  & \textbf{Model 1}  & \textbf{Model 2}  & \textbf{Model 3} & \textbf{Model 4} & \textbf{Model 5} & \textbf{Model 6}\\ \hline
Total cost   & 84168 & 185660 & 31050 & 14250 & 93600 & 15790\\ \hline
Females population at final time    & 0  & 171  & 0 & 0  & 3.2291  & 0  \\ \hline 
Males population at final time   &  0.0028 & 180  & 0.0002  & 0  & 3.2966  & 0  \\ \hline
Female Approximate Eradication Time & 60 & /  & 2  & 7  & /  & 9  \\ \hline
Male Approximate Eradication Time   &  110   & / & 73  & 7  & /  & 8\\ \hline
\end{tabular}}
\label{comp_all}
\end{table}

\section{Conclusions}

In this paper we derive and analyze optimal controls for TYC strategy. The population can be driven to extinction with an optimal $\mu(t)$ that requires a large initial introduction of YY super males. Our simulations, see Fig. \ref{fig:fig_1A}, show that one requires to introduce supermales at a continuous rate of 12-14$\%$ of the carrying capacity for at least a month, before this introduction can catapult the population densities toward controllable levels, for which extinction can occur. 
The difficulty in creating the supermale population motivates alternative approaches, such as harvesting, that attempt to mirror the TYC strategy. However, the effectiveness of driving a population to extinction seems to depend delicately on the form of harvesting. With this in mind we introduce two classes of models, 

\begin{itemize}
\item Harvesting females while stocking males - FHMS models. Herein we introduce three subclasses, linear harvesting, saturating density dependent harvesting and unbounded density dependent harvesting.

\item Harvesting females and harvesting males - FHMH models. Herein again we introduce three subclasses, linear harvesting, saturating density dependent harvesting and unbounded density dependent harvesting.
\end{itemize}

Both classes of models can yield extinction - and are generally more effective than the TYC strategy. To this end we compare them via a rigorous optimal control approach, where our metrics of judging how good or bad a model is, by looking at various criteria such as (1) costs of putting in the controls (2) time it takes to eradicate females (3) time it takes to eradicate males. Also note, we use population parameters that are actual outputs from mesocosm experiments at the USGS laboratory in Gainesville, Florida. To the best of our knowledge this is the first study in the literature which attempts to determine the efficay of the TYC strategy, where best fit parameters from population experiments have been used to tune the models.
Based on these criteria we summarize some of our key results.

Fig.~\ref{model4} indicates that extinction occurs with our optimal controls of Model 4. Notice that a large amount of harvesting must first occur to ensure extinction occurs.  This is reasonable to expect since in order for the population to be attracted to the extinction state then the populations must be driven through substantial harvesting to low enough densities. Model 2 and Model 5, see Fig. \ref{model5}, are not able to achieve eradication (for the parameters considered), although Model 5 female and male fish numbers are lower at the terminal time. Hence, this indicates that if harvesting is modeled by rational functions, that is at some saturating rate, then using this form of harvesting is disadvantageous.  Furthermore, the total costs of Model 2 and 5 are \emph{185660} and \emph{93600}, respectively, indicating the hybrid harvesting strategy FHMS, which mirrors the TYC strategy is less effective as an eradication and control tool than its counterpart FHMH models, for such saturating forms of harvesting.

A similar situation is observed in comparing Model 6 to Model 3.  Fig.~\ref{model6} indicates that extinction can occur when optimal harvesting controls are used. Again, the total costs of Models 3 and 6 are \emph{31050} and \emph{15790}, respectively which indicates that FHMH model is quantifiably better than FHMS model, even when we have unbounded harvesting functions. 
Fig.~\ref{model_2_1} implies that extinction cannot occur if harvesting obeys the saturating form found in Model 2.  However, females and males can be driven to extinction with our optimal controls for Model 3, depicted in Fig.~\ref{model_3_1}, where we use unbounded harvesting functions. This is further evidence that harvesting via a saturating harvesting term such as in Model 2 indicates unfavorable results. However, if harvesting is modeled by a power function then favorable results are found.  In fact, the total cost of \emph{31050} is lower for Model 3 than that of Model 1 - and so we can conclude that within the class of FHMS models, Model 3 is the best; the total cost of \emph{14250} is lower for Model 4 than that of Model 6 - and so we can conclude that within the class of FHMH models, Model 4 is the best. Also we can probably say the class of FHMH models is better than the class of FHMS models. 

There is one \emph{very interesting} exception to this. If we compare Model 6 to Model 3, we see that although Model 6 has a lower cost, females approximately eradicate at the $2^{nd}$ month in the Model 3, whereas it takes about 9 months to make females extinct in the Model 6. Thus if we are concerned with female eradication speed, for the realistic parameters that we have gotten out of our population experiments, Model 3 actually does better than Model 6. 

Note, in various situations the efficacy of harvesting can be indeed questioned.
For example, there is a large body of literature that points to various changes in population parameters as well as body size in fish, under pressure of harvesting. See \cite{sys11}, where it is observed that Lake trout populations  exhibit age reduction at first maturity and clear increases in fecundity when harvested. Also, lake trout body size increased in populations where exploitation caused density reduction \cite{sys11}. Thus in reality, there is strong evidence that fishes will exhibit a density-dependent response to harvest, as seen in these examples - questioning the efficacy of harvest, unless models show otherwise. Clearly their life history parameters change, as their density is affected by harvest, resulting in say increased fecundity. 
One then must consider density dependent life history parameters, such as $\delta, \beta$, when we consider our harvesting class of models. To this end we have begun mesocosm experiments to try and derive the right forms of density dependence under harvesting. Modeling harvesting processes via such density dependent population parameters is our next immediate future goal. The use of harvesting is restricted to certain field situations where the target organisms can be encountered, detected and removed in a practical manner.  Thus, harvesting may not be possible in some situations, such as in vast landscapes that are difficult or impossible to access or when working with organisms that are cryptic or difficult to locate.  Thus, we acknowledge that while it may not be possible in all habitats or with all species, in cases where harvesting is feasible it may be an efficient strategy to reduce or eliminate non-native species populations. Furthermore there is literature on selective harvesting, such as in harvesting females versus males, and its population consequences \cite{MN07}.
Another question is what our results predict about the use of the TYC strategy as an eradication tool. This is very subjective and depends on the resources available to the manager. If one looks at the average rate of harvesting, in our optimally-controlled scenarios it seems the average rate of harvest can vary between $0.03$ for Model 1 (only harvest females) to $0.125$ for Model 4 (harvest both females and males). Note, the natural death rate of the fish ($\delta$) is $0.065$. Thus it really depends if one \emph{can} harvest at these high rates - and what the associated costs are to sustain such a high harvesting rate - again calling into question the efficacy of the harvesting strategy. For example, Model 4 clearly seems the best model in terms of cost and eradication time, but what we don't account for here is the cost of harvesting at approx. $12\%$, when the natural death rate is about $6\%$ - and how feasible this is in the first place. A viable alternative, depending again on what is feasible from a management point of view, could be to try various \emph{hybrid} models, that invoke TYC type dynamics coupled with a certain amount of harvesting. Comparing such hybrid models to pure TYC or pure harvesting strategies, are among some of our future directions.

In essence these harvesting models are a means of mimicking the effect of the TYC strategy, without  the use of the YY super males. The TYC strategy remains a powerful method for invasive species control \cite{GutierrezTeem06, TGP13, Schill2017, Schill16, Schill18}. Our future goal is to use these results to guide further experiments into harvesting strategies. In a closed mesocosm environment this is easily doable as we can always count total populations of the fish periodically. These counts can then be used to harvest/restock exactly those many fish based on the form of harvesting to be used. If favorable results are obtained in the mesocosm, this becomes a confidence booster for natural resource managers to try TYC type strategies coupled with some harvesting, or hybrid strategies for the control of non-native species in the wild. Our results also establish a framework where via harvesting in mesocosm settings - we could actually predict the efficacy of the powerful TYC strategy in the wild.



\section{Figure Legends}


 Figure 1:
\newline
 Left panel: the female/male density change with time without TYC strategy. Right Panel: the female/male density change with time by introducing YY supermales (TYC strategy).

 Figure 2:
\newline
 Male (top) and female (bottom) fancy guppy (\textit{Poecilia reticulata}). Photos were taken by Howard Jelk, USGS.

 Figure 3:
\newline
 Left panel: This is a tank from the USGS facility where the mesocosm experiment was run with fancy guppies (\textit{Poecilia reticulata}). Right panel: preliminary population data over 11-month period. The smooth curve is a result of simulating the basic population model \eqref{TYC_1} (with $s=0$) with best-fit parameters $\beta=0.0057, ~\delta=0.0648,$ and $K=405$.

 Figure 4:
\newline
 (A) Female (red), male (green) and supermale (blue) densities and optimal control $\mu(t)$ in \eqref{TYC_1} change with time $t$. (B) Female (red) and male (green) densities and optimal controls $\eta_{1}(t)$ and $\eta_{2}(t)$ in \eqref{TYC_3_OCT} change with time $t$.

 Figure 5:
\newline
 (A) In this simulation we look at the objective function towards the optimal implement for each model, we clearly see the cost function of TYC model (Model 0) is much larger than FHMS model (Model 1). (B) Since the constraints and scales for $\mu$ and $\eta_{1}, \eta_{2}$ are different, it's hard to compare the effectiveness for each strategy, so here we also look at the cost deducting the controls, that is, $\int^{T}_{0} f+m$  $dt $ under optimal control.

 Figure 6:
\newline
 (A) Female (red) and male (green) densities change with time $t$ with optimal control $\eta_1^*$ (blue) and $\eta_2^*$ (red) in Model 4. As we can see the controls and densities are indistinguishable. (B) Objective function $J_4$ decrease as increasing time with optimal control.

 Figure 7:
\newline
 (A) Female (red) and male (blue) densities change with time $t$ with optimal control $\eta_1^*$ (blue) and $\eta_2^*$ (red) in Model 5. (B) Objective function $J_5$ varies as increasing time with optimal control.

 Figure 8:
\newline
(A) Female (red) and male (green) densities over time $t$ with optimal control $\eta_1^*$ (blue) and $\eta_2^*$ (red) in Model 6. We can see that the densities and controls are indistinguishable. (B) Objective function $J_6$ decrease as increasing time with optimal control.

Figure 9:
\newline
Female (red) and male (green) densities change with time $t$ and the optimal controls $\eta_1^*$ (blue) and $\eta_2^*$ (red) for (A) Model 2 and (B) Model 3. This figure can be found in appendix.

Figure 10:
\newline
Objective function $J_2$ and $J_3$ change with time $t$  with optimal controls $\eta_1^*$ and $\eta_2^*$ for (A) Model 2 and (B) Model 3. This figure is in appendix.

\section{Acknowledgements}

JL and RP would like to acknowledge valuable support from the NSF via DMS-1715377 and DMS-1839993. MB would like to acknowledge valuable support from the NSF via DMS-1715044. U.S. Geological Survey researchers were supported by the Greater Everglades Priority Ecosystem Sciences program. Any use of trade, firm or product names is for descriptive purposes only and does not imply endorsement by the U.S. Government. 

\section{Appendix A: Stability Analysis of classical TYC model}
\label{app}
\subsection{Equilibrium and Stability Analysis when $\boldsymbol{\mu>0}$}
Let $(f^*,m^*,s^*)$ represent an equilibrium of model \eqref{TYC_1}. We are interested in the region $\Omega = \{(f,m,s)~|~ 0 \leq f \leq K,~ 0 \leq m \leq K,~ 0 \leq s \leq K,~ 0 \leq f+m+s \leq K \}$. The dynamics are investigated for the model with a positive introduction rate, $\mu>0$.

\subsubsection{Equilibrium analysis}  An equilibrium with $f^*=0$ denotes a steady state of model \eqref{TYC_1} at which the wild-type XX females are eradicated. There is only one equilibrium that satisfies $f^*=0$, that is, $\left(0,0,\dfrac{\mu}{\delta}\right)$. 

In contrast, assume $f^*>0$ then by direct calculation the equilibrium satisfies
\beq
\label{jj_2}
\begin{split}
& f^*=K-m^*-s^*-\frac{2K \delta}{m^* \beta},\\
& s^*=\frac{\mu}{\delta},
\end{split}
\eeq
where $m^*$ satisfies
\begin{equation}
\label{jj_3}
a m^3+b m^2+c m+d=0.
\end{equation}
The constants $a,~ b,~ c$ and $d$ are given by
\beq
\label{jj_4}
\begin{split}
& a=a(s^*)=2 \beta, \\
& b=b(s^*)=3\beta s^*-K\beta, \\
& c=c(s^*)=2\beta s^{*2}+2K\delta-2K\beta s^*,\\
& d=d(s^*)=4K\delta s^*.
\end{split}
\eeq
Clearly $a, d>0$ since $\mu, ~\beta$, and $K$ are positive. Let 
\beqn
s_c^{+}&=& \frac{K}{2}+\sqrt{\frac{K^2}{4}-\frac{K\delta}{\beta}},\\
s_c^{-}&=& \frac{K}{2}-\sqrt{\frac{K^2}{4}-\frac{K\delta}{\beta}},\\
s_b    &=& \frac{K}{3}.
\eeqn
It is easy to check that $s_c^{\pm}$ and $s_b$ are the nonzero roots for $c(s^*)=0$ and $b(s^*)=0$, respectively.  Define $\alpha=\frac{K \beta}{\delta}$. For illustrative purposes, the population data provides the following estimates: $\alpha \approx 30.375$, $s_c^+ \approx 404.925$, $s_c^-\approx .075014$, and $s_b \approx 135$.  

\begin{lemma} ~ \label{lm_1}
\begin{enumerate}[label=(\alph*)]
\item If $\alpha>\frac{9}{2}$ then $s_c^{-}<s_b<s_c^{+}$;

\item If $4<\alpha<\frac{9}{2}$ then $s_b<s_c^{-}<s_c^{+}$;

\item If $\alpha=4$ then $s_b< s_c^{-}=s_c^{+}$;

\item If $\alpha=\frac{9}{2}$ then $s_b=s_c^{-}<s_c^{+}$.
\end{enumerate}
\end{lemma}

\begin{proof}
It is easy to be verified by direct calculations.
\end{proof}

Next, consider the discriminant of \eqref{jj_3}, that is,
$$ \Delta=b^2 c^2-4a c^3 -4b^3 d-27 a^2 d^2+18 a b c d. $$
\begin{proposition}
\label{prop_1}
Assume that $ \Delta >0$. 
\begin{enumerate}[label=(\alph*)]
\item If $\alpha>\frac{9}{2}$, the \eqref{jj_3} has  2 positive roots when $ \frac{\mu}{\delta}<s_c^+ $ and no positive root when $\frac{\mu}{\delta} \geq s_c^+.$ 

\item If $s^* \in [s_b, s_c^-] \cup [s_c^+, \infty)$, then \eqref{jj_3} has 2 positive roots when $\frac{\mu}{\delta}  \in (0, s_b) \cup (s_c^-, s_c^+)$ and no positive root when $\frac{\mu}{\delta}  \in [s_b, s_c^-] \cup [s_c^+, \infty)$. 
\end{enumerate}

\end{proposition}

\begin{proof}

Define
\begin{equation}
\label{jj_5}
h(m)=a m^3+b m^2+c m+d.
\end{equation}
Since $a>0, d>0$ and $\Delta>0$, the cubic function \eqref{jj_5} has 3 distinct nonzero real roots. 

\begin{enumerate}
\item [(a)] If $\alpha>\frac{9}{2}$, by lemma \eqref{lm_1}, we have $s_c^{-}<s_b<s_c^{+}$.

\begin{enumerate}[label=(\roman*)]
\item Assume $s^*=\frac{\mu}{\delta} \in (0, s_c^-)$. Then $b<0$ and $c>0$. The signs of coefficients of $h(m)$ and $h(-m)$ are 
\beqn
\label{jj_6}
\mbox{sign}(a,b,c,d)&=&(+,-,+,+), \\
\label{jj_7}
\mbox{sign}(-a,b,-c,d)&=&(-,-,-,+).
\eeqn
Thus, two sign changes in $h(m)$ and one sign changes in $h(-m)$. By Descartes' rule of signs, $h(m)=0$ has either 2 or 0 positive roots and exactly 1 negative root. Since $h(m)=0$ has three distinct nonzero real roots, the roots of \eqref{jj_5} has to be 2 positive roots and 1 negative root. 

\item Assume $s^*=\frac{\mu}{\delta} \in [s_c^-, s_b)$. If $s^*=s_c^-$, then $c=0, b<0$; if $s^{*} \in (s_c^-, s_b)$, then $c<0, b<0$. By Descartes' rule of signs, it implies $h(m)$ has 2 positive roots and 1 negative one. 

\item Assume $s^*=\frac{\mu}{\delta} \in [s_b, s_c^+)$. If $s^*=s_b$, then $b=0, c<0$; if $s^* \in (s_b, s_c^+)$, then $b>0, c<0$. By a similar argument, $h(m)$ has 2 positive roots and 1 negative root.

\item Assume $s^*=\frac{\mu}{\delta} \in [s_c^+, \infty)$.  Consequently, $b,c>0$. We have no positive roots for $h(m)$ because there is no sign change in the coefficient of $h(m)$.
\end{enumerate}

\item [(b)] If $\alpha \leq \frac{9}{2}$, we always have $s_b \leq s_c^{-}$. Similarly, 
\begin{enumerate}[label=(\roman*)]
\item If $s^* \in (0, s_b) \cup (s_c^-, s_c^+)$, then either $b<0, c>0$ or $b>0, c<0$. There are two sign changes for the coefficient of $h(m)$, therefore $h(m)$ has 2 positive roots. 

\item If $s^* \in [s_b, s_c^-] \cup [s_c^+, \infty)$, $b \geq 0, c \geq 0$, there is no sign change for the coefficient of $h(m)$ and therefore $h(m)$ has no positive root. 
\end{enumerate}
\end{enumerate}

\end{proof}

\begin{proposition}
\label{prop_2}
Assume that $\Delta=0$.
\begin{enumerate}[label=(\alph*)]
\item If $\alpha>\frac{9}{2}$, then \eqref{jj_3} has 2 positive roots when $ \frac{\mu}{\delta}<s_c^+ $ and no positive root when $\frac{\mu}{\delta} \geq s_c^+.$ 

\item If $s^* \in [s_b, s_c^-] \cup [s_c^+, \infty)$, then \eqref{jj_3} has 2 positive roots when $\frac{\mu}{\delta}  \in (0, s_b) \cup (s_c^-, s_c^+)$ and no positive root when $\frac{\mu}{\delta}  \in [s_b, s_c^-] \cup [s_c^+, \infty)$. 
\end{enumerate}
\end{proposition}

\begin{proof}
If $\Delta=0, h(m)=0$ has a repeated root and all its roots are real. The results follow by Descartes' rule of signs. 
\end{proof}

\begin{proposition}
\label{prop_3}
Assume that $\Delta<0$. The equation \eqref{jj_3} has no positive root. 
\end{proposition}

\begin{proof}
If $\Delta=0,$ then $h(m)=0$ has one real root and two non-real complex conjugate roots. The results follow by Descartes' rule of signs. 
\end{proof}

\subsubsection{Stability analysis of equilibria}
The associated Jacobian matrix of the model \eqref{TYC_1} is given by 
\begin{equation} \label{Jac_1} 
\mathrm{J}=\left[ \begin{array}{ccc}
J_ {11}& J_{12}  & J_{13} \\ 
J_{21} & J_{22} & J_{23} \\ 
0 & 0 & J_{33}  \end{array}
\right] 
\end{equation} 
where
\beqn
\label{jj_8}
J_{11} &=& \frac{1}{2} m^* \beta L^* -\frac{1}{2}  \frac{f^* m^* \beta}{K} -\delta, \\
\label{jj_9}
J_{12} &=& \frac{1}{2} f^* \beta L^* - \frac{1}{2} \frac{f^* m^* \beta}{K}, \\
\label{jj_10}
J_{13} &=& - \frac{1}{2} \frac{f^* m^* \beta}{K}, \\
\label{jj_11}
J_{21} &=& (\frac{1}{2}m^*+s^*) \beta L^*-\frac{(\frac{1}{2}f^* m^*+f^* s^*)\beta}{K}, \\
\label{jj_12}
J_{22} &=&\frac{1}{2}f^*\beta L^* -\frac{(\frac{1}{2}f^* m^*+ f^* s^*)\beta}{K}-\delta, \\
\label{jj_13}
J_{23} &=&f^* \beta L^* -\frac{(\frac{1}{2}f^* m^*+ f^* s^*)\beta}{K}, \\
\label{jj_14}
J_{33} &=&-\delta,
\eeqn
and
\beq
\label{jj_15}
L^*=1-\frac{f^*+m^*+s^*}{K}.
\eeq

The corresponding characteristic function is given by 
\beq
\label{jj_16}
\lambda^3+k_{12} \lambda^2+ k_{11} \lambda+ k_{10}=0
\eeq
with 
\beq
\label{jj_17}
\begin{split}
& k_{12}=-(J_{11}+J_{22}+J_{33}),\\
& k_{11}=J_{11} J_{22}+J_{11}J_{33}-J_{12}J_{21}+J_{22}J_{33},\\
& k_{10}=J_{12}J_{21}J_{33}-J_{11}J_{22}J_{33}.
\end{split}
\eeq

\begin{theorem}
\label{thm_1}
Let $\mu>0$. The interior equilibrium $(f^*, m^*, \frac{\mu}{\delta})$ where $f^*,m^*>0$ is locally stable if 
\begin{equation}
k_{10}>0,~ k_{11}>0,~ k_{12}>0,~ k_{11}k_{12}-k_{10}>0.
\end{equation}
\end{theorem}

\begin{proof}
It follows from Routh Hurwitz stability criteria. 
\end{proof}

\section{Appendix B: Nonlinear Forms of Harvesting for FHMS Models}
\label{AppendixB}
In the previous two sections linear harvesting functions considered only. Here, we compare the previous results to nonlinear harvesting functions. That is, we examine the control on the system
\begin{equation}
\label{PH_0}
\begin{split}
& \frac{df}{dt}=\frac{1}{2}fm\beta L-\delta f-\eta_1 G_1(f),\\
& \frac{dm}{dt}=\frac{1}{2}fm\beta L-\delta m+\eta_2 G_2(m),
\end{split}
\end{equation}
where $G_1(f)$ and $G_2(m)$ are non-negative, possibly nonlinear, functions such that $\{f,m\} \in \{0 \leq f,m \leq K$ and $0 \leq f+m \leq K \}$. In the previous section $G_1(f)=f$ and $G_2(m)=m$; that model will be called \textit{Model 1} in the forthcoming discussion. 


\subsection{Model 2: Nonlinear Harvesting when $\boldsymbol{G_1(f)=\frac{f}{f+d_1},~ G_2(m)=\frac{m}{m+d_2}}$}
Consider the model
\begin{equation}
\label{HP_2}
\begin{split}
& \frac{df}{dt}=\frac{1}{2}fm\beta L-\delta f-\eta_1 \frac{f}{f+d_1},\\
& \frac{dm}{dt}=\frac{1}{2}fm\beta L-\delta m+\eta_2 \frac{m}{m+d_2},
\end{split}
\end{equation}
where $d_1>0, d_2>0$.  The nonlinear harvesting function are rational functions and attempt to model a saturation of harvesting at large populations.  

\begin{theorem}
\label{HP2_global}
The trivial equilibrium of \eqref{HP_2} is globally asymptotically stable if $\beta K<2\delta+\frac{\eta_1}{K+d_1}-\frac{\eta_2}{d_2}$ and $\delta>\frac{\eta_2}{d_2}$.
\end{theorem}
\begin{proof}
Consider the Lyapunov function $V=f+m$. It is left to show that $\frac{dV}{dt}<0$ for all $(f,m) \not =(0,0)$. Consider 
\beq
\begin{split}
\frac{dV}{dt}&=\frac{df}{dt}+\frac{dm}{dt}\\
&=fm\beta L-(\delta+\frac{\eta_1}{f+d_1})f-(\delta-\frac{\eta_2}{m+d_2})m\\
&\leq fm\beta-(\delta+\frac{\eta_1}{f+d_1})f-(\delta-\frac{\eta_2}{m+d_2})m, \quad \mathrm{since} \quad L \leq 1 \\
&\leq fm\beta-(\delta+\frac{\eta_1}{K+d_1})f-(\delta-\frac{\eta_2}{d_2})m\\
&\leq [\beta-\frac{1}{m}(\delta+\frac{\eta_1}{K+d_1})-\frac{1}{f}(\delta-\frac{\eta_2}{d_2})] fm\\
&\leq [\beta-\frac{1}{K}(\delta+\frac{\eta_1}{K+d_1})-\frac{1}{K}(\delta-\frac{\eta_2}{d_2})] fm.\\
\end{split}
\eeq
It is enough to show $\beta-\frac{1}{K}(\delta+\frac{\eta_1}{K+d_1})-\frac{1}{K}(\delta-\frac{\eta_2}{d_2})<0$.  Now $\delta-\frac{\eta_2}{d_2}>0$ and, by direct calculation, $\beta K<2\delta+\frac{\eta_1}{K+d_1}-\frac{\eta_2}{d_2}$.  Subsequently, $\frac{dV}{dt}<0.$
\end{proof}

As in the previous section, assume $\eta_1, \eta_2$ are time dependent and set the objective function as 
\beq
  J_2(\eta_{1},\eta_{2})= \int^{T}_{0} -(f+m)- \frac{1}{2}(\eta_{1}^{2} +\eta_{2}^{2} ) dt,
\eeq
subject to \eqref{HP_2} and with initial conditions $f(t_{0})=f_{0}, m(t_{0})=m_0$.  Optimal controls are sought within the set $U_2$ where
\beq
 U_2= \{ (\eta_{1},\eta_{2})~|~\eta_{i} \ \mbox{measurable}, \  0 \leq \eta_{1} \leq 1, 0 \leq \eta_{2} \leq 1, \  t \in [0,T], \ \forall T\}.
\eeq
The goal is to seek an optimal $(\eta_{1}^{*}, \eta_{2}^{*})$ such that
\beq
\begin{split}
  J_2(\eta_{1}^{*}, \eta_{2}^{*}) &= \underset{(\eta_{1},\eta_{2})}{\max} \int^{T}_{0} -(f+m) - \frac{1}{2}(\eta_1^2+\eta_{2}^{2})  dt \\
  & = \underset{(\eta_{1},\eta_{2})}{\min} \int^{T}_{0} f+m + \frac{1}{2}(\eta_1^2+\eta_{2}^{2})  dt.
 \end{split}
\eeq
Consider the following existence theorem.

\begin{theorem}
\label{PH2_thm_oct}
Consider the optimal control problem \eqref{HP_2}. There exists $(\eta_{1}^{*}, \eta_{2}^{*}) \in U_2$ such that
\beq
  J_2(\eta_{1}^{*}, \eta_{2}^{*}) = \underset{(\eta_{1},\eta_{2})}{\max} \int^{T}_{0} -(f+m) - \frac{1}{2}(\eta_1^2+\eta_{2}^{2})  dt.
\eeq
\end{theorem}

\begin{proof}
The proof is similar to Theorem \ref{thm_tyc_3_oct} and is omitted for brevity.
\end{proof}

Consider the Hamiltonian for $J_2$,
\beq
\label{HP2_ham}
H_2=-(f+m) - \frac{1}{2}(\eta_1^2+\eta_{2}^{2}) +\lambda_1 f'+\lambda_2 m'.
\eeq
The necessary conditions on the optimal controls are derived by applying Pontryagin's maximum principle. Namely,
\beq
\begin{split}
\lambda_1'(t) =& 1-\lambda_1 [\frac{m \beta }{2} (1-\frac{f+m}{K})-\frac{f m \beta}{2K}-\delta-\frac{\eta_1}{f+d_1}+\frac{\eta_1 f}{(f+d_1)^2}] \\
& -\lambda_2 [\frac{ m \beta}{2} (1-\frac{f+m}{K})-\frac{ f m \beta}{2K}],\\
\lambda_2'(t)= &1-\lambda_1 [\frac{f \beta }{2} (1-\frac{f+m}{K})-\frac{f m \beta}{2K}] \\
&-\lambda_2 [\frac{ f \beta}{2} (1-\frac{f+m}{K})-\frac{ f m \beta}{2K}-\delta+\frac{\eta_2}{m+d_2}-\frac{\eta_2 m}{(m+d_2)^2}],
\end{split}
\eeq
with the transversality condition given as $\lambda_1(T)=\lambda_2(T)=0$. Differentiating with respect to our controls yields,
\beq
\begin{split}
& \frac{\partial H_2}{\partial \eta_1}=-\frac{f\lambda_1}{f+d_1}-\eta_1,\\
& \frac{\partial H_2}{\partial \eta_2}=\frac{m\lambda_2}{m+d_2}-\eta_2.
\end{split}
\eeq
Hence, the optimal $\eta_1^*,$ is characterized by the three cases:

\begin{enumerate}
\item If $\frac{\partial H_2}{\partial \eta_{1}}<0$ then $\eta_{1}^{*}=0$.  This implies that $- \frac{f\lambda_1}{f+d_1}<0;$
\item If $\frac{\partial H_2}{\partial \eta_{1}}=0$ then $\eta_{1}^{*}=-\frac{f\lambda_1}{f+d_1}$.  This implies that $0 \leq - \frac{f\lambda_1}{f+d_1} \leq 1$;
\item If $\frac{\partial H_2}{\partial \eta_{1}}>0$ then $\eta_{1}^{*}=1$.  This implies that $1<- \frac{f\lambda_1}{f+d_1}.$
\end{enumerate}

Similar results are established for the optimal $\eta_2^*$. A compact way of writing the optimal control $\eta_1, \eta_2$ is 
\begin{eqnarray}
\label{HP2_1}
\eta_1^{*}&=&\min\left(1, \max\left(0, -\frac{f\lambda_1}{f+d_1}\right)\right)\\
\label{HP2_2}
\eta_2^{*}&=&\min\left(1, \max\left(0, \frac{ m \lambda_2}{m+d_2}\right)\right).
\end{eqnarray}

\begin{theorem}
An optimal control $(\eta_1^*, \eta_2^*) \in U_2$ for the system \eqref{HP_2} that maximizes the objective functional $J_2$ is characterized by \eqref{HP2_1}-\eqref{HP2_2}.
\end{theorem}

\subsection{Model 3: Nonlinear Harvesting when $\boldsymbol{G_1(f)=f^{\frac{3}{2}}, ~G_2(m)=m^{\frac{3}{2}}}$}

Consider the the model
\begin{equation}
\label{HP_3}
\begin{split}
& \frac{df}{dt}=\frac{1}{2}fm\beta L-\delta f-\eta_1 f^{\frac{3}{2}},\\
& \frac{dm}{dt}=\frac{1}{2}fm\beta L-\delta m+\eta_2 m^{\frac{3}{2}}.
\end{split}
\end{equation}
The nonlinear harvesting functions attempt to model a sharp increase in the ability to harvest fish at low populations.  Of course, in contrast to the previous Model 2, for large populations the harvesting terms has no upper bound. Such terms could express chemical use that are density dependent \cite{SL15}. In order to compare  Models 2 and 3 see Fig. \ref{model_3_2} - Fig. \ref{model_4_2}.


\begin{theorem}
\label{thm_global_HP3}
The trivial equilibrium is globally asymptotically stable in \eqref{HP_3} if $\beta K<2\delta-\eta_2 \sqrt{K}$ and $\delta>\eta_2 \sqrt{K}$.
\end{theorem}
\begin{proof}
Consider the Lyapunov function $V=f+m$. It remains to show that  $\frac{dV}{dt}<0$ for all $(f,m) \not =(0,0)$. Consider 
\beq
\begin{split}
\frac{dV}{dt}&=\frac{df}{dt}+\frac{dm}{dt}\\
&=fm\beta L-(\delta+\eta_1 \sqrt{f})f-(\delta-\eta_2 \sqrt{m})m\\
&\leq fm\beta -(\delta+\eta_1 \sqrt{f})f-(\delta-\eta_2 \sqrt{m})m, \quad \mathrm{since} \quad L \leq 1 \\
&\leq fm\beta-\delta f-(\delta-\eta_2 \sqrt{K})m\\
&\leq (\beta-\frac{\delta}{m}-\frac{(\delta-\eta_2 \sqrt{K})}{f}) fm\\
&\leq  (\beta-\frac{\delta}{K}-\frac{(\delta-\eta_2 \sqrt{K})}{K}) fm.\\
\end{split}
\eeq
It is enough to show $\beta-\frac{\delta}{K}-\frac{(\delta-\eta_2 \sqrt{K})}{K}<0$ and $\delta-\eta_2 \sqrt{K}>0$. By direct calculation, we have $\beta K<2\delta-\eta_2 \sqrt{K}$ and $\delta>\eta_2 \sqrt{K}$, which complete the proof.
\end{proof}

As in the previous section, we assume $\eta_1$ and $\eta_2$ are time dependent and set the objective function
\beq
  J_3(\eta_{1},\eta_{2})= \int^{T}_{0} -(f+m)- \frac{1}{2}(\eta_{1}^{2} +\eta_{2}^{2} ). dt
\eeq
subject to \eqref{HP_3} and with initial conditions $f(t_{0})=f_{0}, m(t_{0})=m_0$. The optimal controls are searched for within the set $U_3$ where
\beq
 U_3= \{ (\eta_{1},\eta_{2})~|~\eta_{i} \ \mbox{measurable}, \  0 \leq \eta_{1} \leq 1, 0 \leq \eta_{2} \leq 1, \  t \in [0,T], \ \forall T\}.
\eeq
The goal is to seek an optimal $(\eta_{1}^{*}, \eta_{2}^{*})$ such that
\beq
\begin{split}
  J_3(\eta_{1}^{*}, \eta_{2}^{*}) &= \underset{(\eta_{1},\eta_{2})}{\max} \int^{T}_{0} -(f+m) - \frac{1}{2}(\eta_1^2+\eta_{2}^{2})  dt \\
  & = \underset{(\eta_{1},\eta_{2})}{\min} \int^{T}_{0} f+m + \frac{1}{2}(\eta_1^2+\eta_{2}^{2})  dt.\\
 \end{split}
\eeq
Again, we prove the following existence theorem.

\begin{theorem}
\label{PH3_thm_oct}
Consider the optimal control problem \eqref{HP_3}. There exists $(\eta_{1}^{*}, \eta_{2}^{*}) \in U_3$ s.t.

\beq
  J_3(\eta_{1}^{*}, \eta_{2}^{*}) = \underset{(\eta_{1},\eta_{2})}{\max} \int^{T}_{0} -(f+m) - \frac{1}{2}(\eta_1^2+\eta_{2}^{2})  dt.
\eeq
\end{theorem}

\begin{proof}
The proof is similar to Theorem \ref{thm_tyc_3_oct} and is omitted for brevity.
\end{proof}

The Hamiltonian for $J_3$ is
\beq
\label{HP3_ham}
H_3=-(f+m) - \frac{1}{2}(\eta_1^2+\eta_{2}^{2}) +\lambda_1 f'+\lambda_2 m'.
\eeq
The necessary conditions on the optimal controls are determined by applying Pontryagin's maximum principle, that is,
\beq
\begin{split}
\lambda_1'(t) =& 1-\lambda_1 [\frac{m \beta }{2} (1-\frac{f+m}{K})-\frac{f m \beta}{2K}-\delta-\frac{3}{2} \eta_1 f^{\frac{1}{2}}] \\
& -\lambda_2 [\frac{ m \beta}{2} (1-\frac{f+m}{K})-\frac{ f m \beta}{2K}],\\
\lambda_2'(t)= &1-\lambda_1 [\frac{f \beta }{2} (1-\frac{f+m}{K})-\frac{f m \beta}{2K}] \\
&-\lambda_2 [\frac{ f \beta}{2} (1-\frac{f+m}{K})-\frac{ f m \beta}{2K}-\delta+\frac{3}{2} \eta_2 m^{\frac{1}{2}}],
\end{split}
\eeq
with the transversality condition of $\lambda_1(T)=\lambda_2(T)=0$. Differentiating the Hamiltonian with respect to the control variables yields,
\beq
\begin{split}
& \frac{\partial H_3}{\partial \eta_1}=-\lambda_1 f^{\frac{3}{2}}-\eta_1,\\
& \frac{\partial H_3}{\partial \eta_2}=\lambda_2 m^{\frac{3}{2}}-\eta_2.
\end{split}
\eeq
Hence, the optimal control for $\eta_1^*$ and $\eta_2^*$ is written compactly as
%
\begin{eqnarray}
\label{HP3_1}
\eta_1^{*}&=&\min(1, \max(0, -\lambda_1 f^{\frac{3}{2}})),\\
\label{HP3_2}
\eta_2^{*}&=&\min(1, \max(0, ~\lambda_2 m^{\frac{3}{2}})).
\end{eqnarray}
\begin{theorem}
An optimal control $(\eta_1^*, \eta_2^*) \in U_3$ for the system \eqref{HP_3} that maximizes the objective functional $J_3$ is characterized by \eqref{HP3_1}-\eqref{HP3_2}.
\end{theorem}

\subsection{Numerical Simulations and Comparisons}
In this section, we will numerically simulate the optimal strategy and its corresponding optimal states for each model. Table \eqref{table_2} details the parameters used in the simulations.
\begin{table}[H]
\caption{Parameters used for numerical simulations}
\label{table_2}
\begin{tabular}{|c|l|c|}
\hline
\textbf{Parameter }            & ~~\textbf{Description}       & \textbf{Value} \\ \hline
$ \beta $  & Birth rate        & 0.005774 \\ \hline
$ \delta $ & Death rate        & 0.0648  \\ \hline
$K$                     & Carrying capacity & 405.0705   \\ \hline
$T$                     & Terminal time     & 200   \\ \hline
$d_1$                  & parameter in $G_1$     & 1  \\ \hline
$d_2$                 & parameter in $G_2$     & 1  \\ \hline
\end{tabular}
\end{table}

\begin{figure}[H]
        \begin{minipage}[b]{0.480\linewidth}
            \centering
            \includegraphics[width=\textwidth]{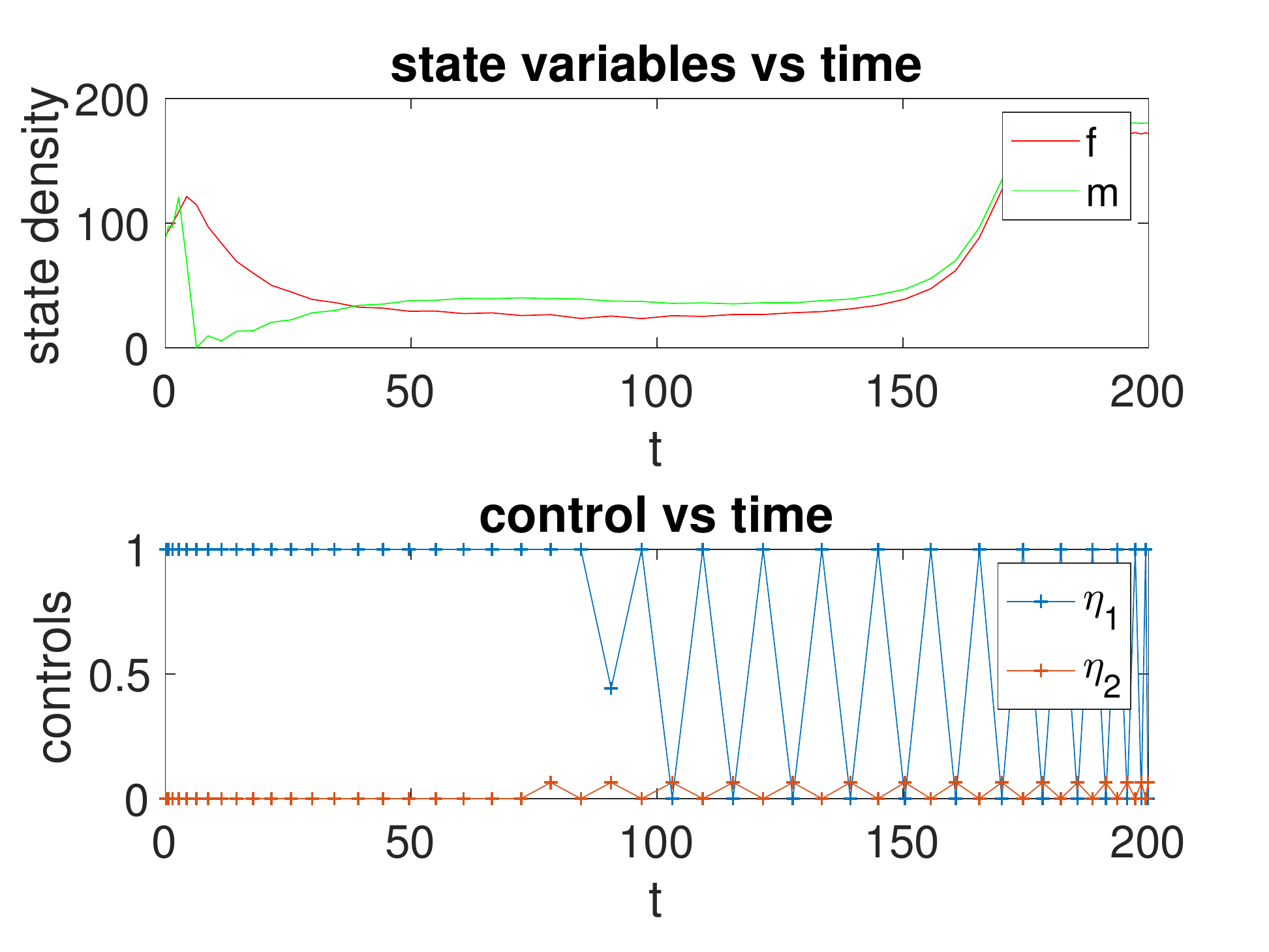}
            \subcaption{Model 2}
             \label{model_2_1}
        \end{minipage}
        \hspace{0.10cm}
        \begin{minipage}[b]{0.480\linewidth}
            \centering
            \includegraphics[width=\textwidth]{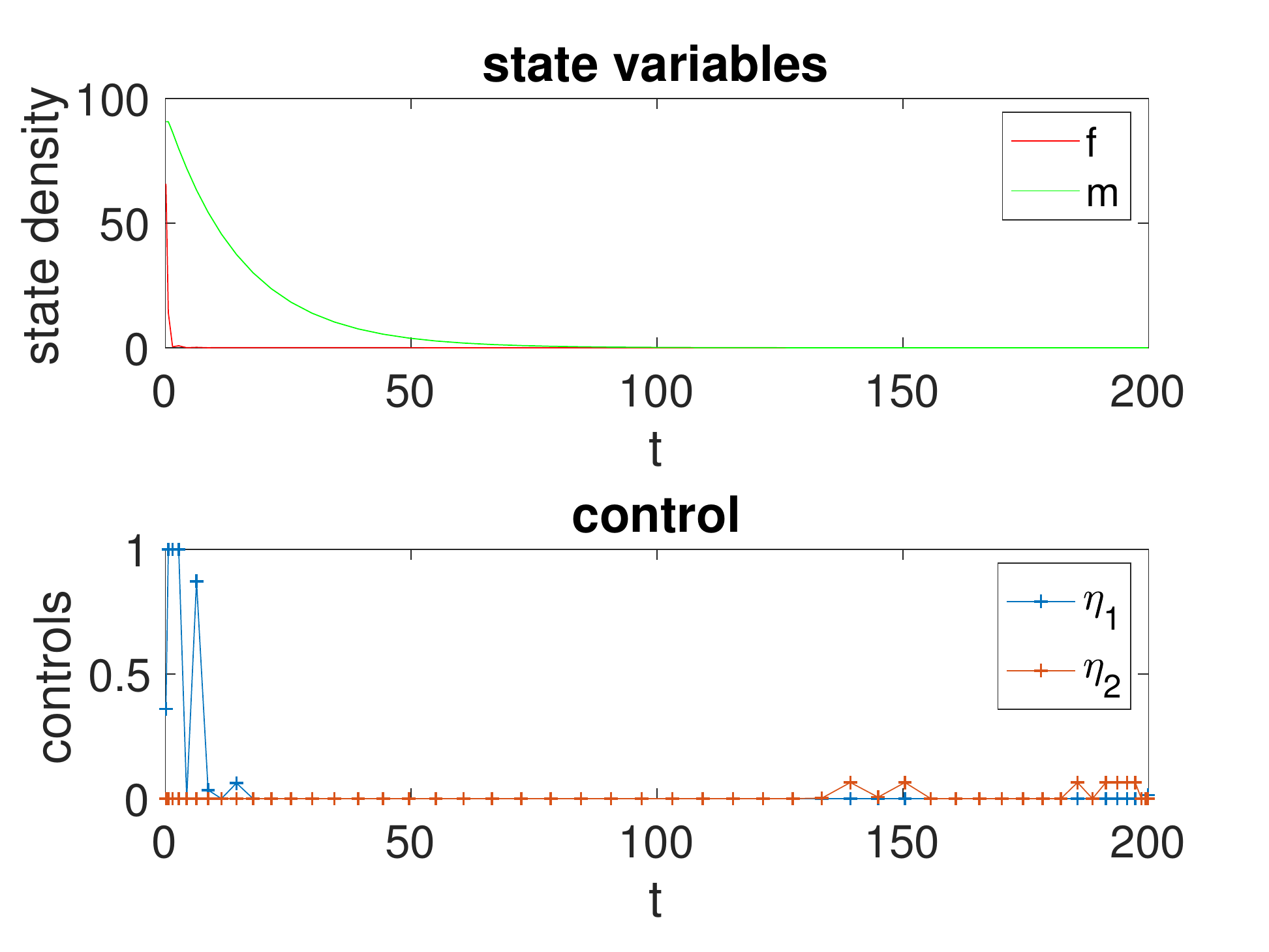}
            \subcaption{Model 3}
             \label{model_3_1}
        \end{minipage}
\caption{Female (red) and male (green) densities change with time $t$ and the optimal controls $\eta_1^*$ (blue) and $\eta_2^*$ (red) for (A) Model 2 and (B) Model 3.  }
\end{figure}

\begin{figure}[H]
\label{model34cost}
        \begin{minipage}[b]{0.480\linewidth}
            \centering
            \includegraphics[width=\textwidth]{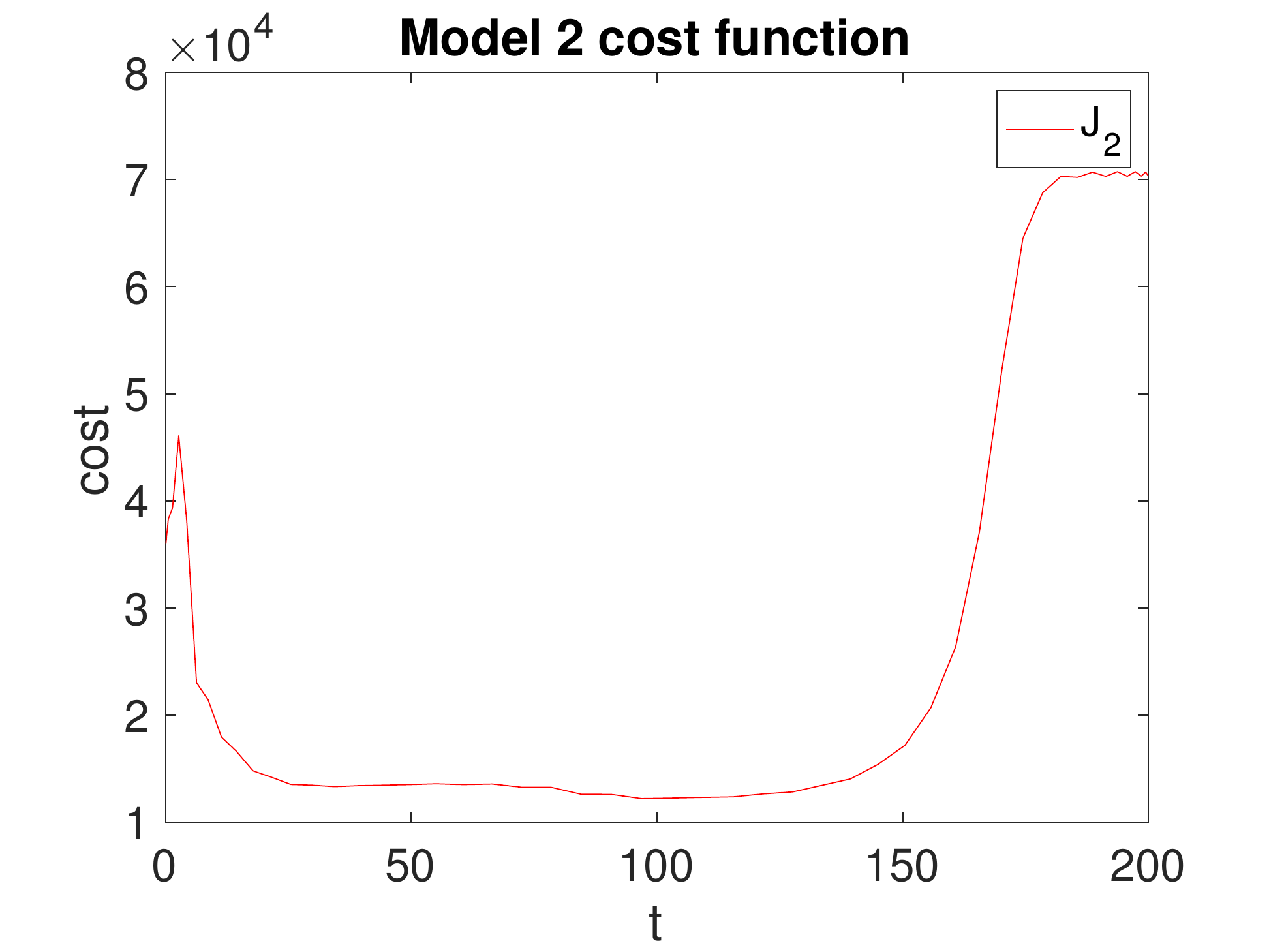}
            \subcaption{ Model 2}
             \label{model_3_2}
        \end{minipage}
        \hspace{0.10cm}
        \begin{minipage}[b]{0.480\linewidth}
            \centering
            \includegraphics[width=\textwidth]{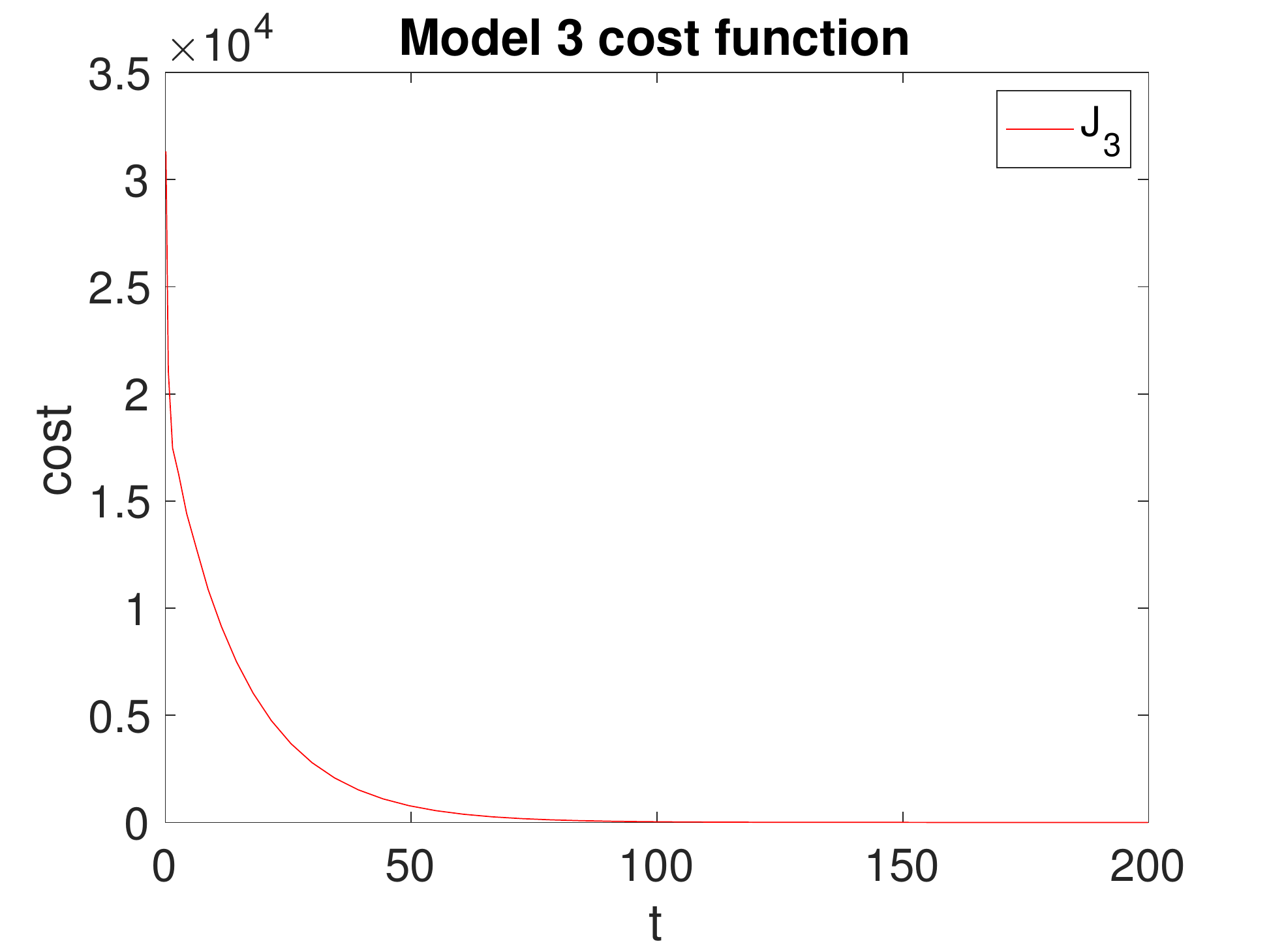}
            \subcaption{ Model 3}
             \label{model_4_2}
        \end{minipage}
\caption{Objective function $J_2$ and $J_3$ change with time $t$  with optimal controls $\eta_1^*$ and $\eta_2^*$ for (A) Model 2 and (B) Model 3.}
\end{figure}

%

\section{Appendix C: Nonlinear forms of Harvesting for FHMH Models}
\label{AppendixC}

\subsection{Model 5: $\boldsymbol{G_1(f)=\frac{f}{f+d_1},G_2(m)=\frac{m}{m+d_2}}$}
The model is given by
\begin{equation}
\label{NHP_2}
\begin{split}
& \frac{df}{dt}=\frac{1}{2}fm\beta L-\delta f-\eta_1 \frac{f}{f+d_1},\\
& \frac{dm}{dt}=\frac{1}{2}fm\beta L-\delta m-\eta_2 \frac{m}{m+d_2},
\end{split}
\end{equation}
where $d_1>0, d_2>0$. Again, the harvesting functions attempt to model a saturation of harvesting for large populations.

\subsubsection{Global stability}
The following theorem provides the criteria for global stability of the trivial equilibrium. 
\begin{theorem}
\label{NHP2_global}
The trivial equilibrium is globally asymptotically stable in \eqref{NHP_2} if $\beta K<2\delta+\frac{\eta_1}{K+d_1}+\frac{\eta_2}{K+d_2}$.
\end{theorem}
\begin{proof}
Consider the Lyapunov function $V=f+m$. It is left to show that $\frac{dV}{dt}<0$ for all $(f,m) \not =(0,0)$. Consider 
\beq
\begin{split}
\frac{dV}{dt}&=\frac{df}{dt}+\frac{dm}{dt}\\
&=fm\beta L-(\delta+\frac{\eta_1}{f+d_1})f-(\delta+\frac{\eta_2}{m+d_2})m\\
&\leq fm\beta-(\delta+\frac{\eta_1}{f+d_1})f-(\delta+\frac{\eta_2}{m+d_2})m, \quad \mathrm{since} \quad L \leq 1 \\
&= [\beta-\frac{1}{m}(\delta+\frac{\eta_1}{f+d_1})-\frac{1}{f}(\delta+\frac{\eta_2}{m+d_2})] fm\\
&\leq [\beta-\frac{1}{m}(\delta+\frac{\eta_1}{K+d_1})-\frac{1}{f}(\delta+\frac{\eta_2}{K+d_2})] fm\\
&\leq [\beta-\frac{1}{K}(\delta+\frac{\eta_1}{K+d_1})-\frac{1}{K}(\delta+\frac{\eta_2}{K+d_2})] fm.\\
\end{split}
\eeq
It is enough to show $\beta-\frac{1}{K}(\delta+\frac{\eta_1}{K+d_1})-\frac{1}{K}(\delta+\frac{\eta_2}{K+d_2})<0$. By direct calculation, we have $\beta K<2\delta+\frac{\eta_1}{K+d_1}+\frac{\eta_2}{K+d_2}$, which proves theorem \eqref{NHP2_global}.
\end{proof}

\subsubsection{Optimal control}
We again assume $\eta_1, \eta_2$ are time dependent and set the objective function as 
\beq
  J_5(\eta_{1},\eta_{2})= \int^{T}_{0} -(f+m)- \frac{1}{2}(\eta_{1}^{2} +\eta_{2}^{2} ) dt.
\eeq
subject to \eqref{NHP_2} and with initial conditions $f(t_{0})=f_{0}, m(t_{0})=m_0$. Optimal controls are sought within $U_5$, that is,
\beq
 U_5= \{ (\eta_{1},~\eta_{2})~|~\eta_{i} \ \mbox{measurable}, \  0 \leq \eta_{1} \leq 1, 0 \leq \eta_{2} \leq 1, \  t \in [0,T], \ \forall T\}.
\eeq
Hence, optimal $(\eta_{1}^{*}, \eta_{2}^{*})$ are sought such that,
\beq
\begin{split}
  J_5(\eta_{1}^{*}, \eta_{2}^{*}) &= \underset{(\eta_{1},\eta_{2})}{\max} \int^{T}_{0} -(f+m) - \frac{1}{2}(\eta_1^2+\eta_{2}^{2})  dt \\
  & = \underset{(\eta_{1},\eta_{2})}{\min} \int^{T}_{0} f+m + \frac{1}{2}(\eta_1^2+\eta_{2}^{2})  dt.\\
 \end{split}
\eeq
The following existence theorem is given.

\begin{theorem}
\label{NPH2_thm_oct}
Consider the optimal control problem \eqref{HP_2}. There exists $(\eta_{1}^{*}, \eta_{2}^{*}) \in U_5$ such that
\beq
  J_5(\eta_{1}^{*}, \eta_{2}^{*}) = \underset{(\eta_{1},\eta_{2})}{\max} \int^{T}_{0} -(f+m) - \frac{1}{2}(\eta_1^2+\eta_{2}^{2})  dt.
\eeq
\end{theorem}

\begin{proof}
The proof is similar to Theorem \ref{thm_NPH1_oct} and is omitted for brevity.
\end{proof}

Consider the Hamiltonian for $J_5$,
\begin{equation}
\label{NHP2_ham}
H_5=-(f+m) - \frac{1}{2}(\eta_1^2+\eta_{2}^{2}) +\lambda_1 f'+\lambda_2 m'.
\end{equation}
Pontryagin's maximum principle is applied to determine necessary conditions on the optimal controls.  The differential equations for the adjoint $\lambda_i$ is
\beq
\begin{split}
\lambda_1'(t) =& 1-\lambda_1 [\frac{m \beta }{2} (1-\frac{f+m}{K})-\frac{f m \beta}{2K}-\delta-\frac{\eta_1}{f+d_1}+\frac{\eta_1 f}{(f+d_1)^2}] \\
& -\lambda_2 [\frac{ m \beta}{2} (1-\frac{f+m}{K})-\frac{ f m \beta}{2K}].\\
\lambda_2'(t)= &1-\lambda_1 [\frac{f \beta }{2} (1-\frac{f+m}{K})-\frac{f m \beta}{2K}] \\
&-\lambda_2 [\frac{ f \beta}{2} (1-\frac{f+m}{K})-\frac{ f m \beta}{2K}-\delta-\frac{\eta_2}{m+d_2}+\frac{\eta_2 m}{(m+d_2)^2}].
\end{split}
\eeq
with the transversality condition $\lambda_1(T)=\lambda_2(T)=0.$ Differentiating the Hamiltonian,
\beq
\begin{split}
& \frac{\partial H_5}{\partial \eta_1}=-\frac{f\lambda_1}{f+d_1}-\eta_1,\\
& \frac{\partial H_5}{\partial \eta_2}=\frac{m\lambda_2}{m+d_2}-\eta_2.
\end{split}
\eeq
%
Hence, the optimal control $\eta_1, \eta_2$ is written compactly as,
\begin{eqnarray}
\label{NHP2_11}
\eta_1^{*}&=&\min\left(1, \max\left(0, -\frac{f\lambda_1}{f+d_1}\right)\right),\\
\label{NHP2_22}
\eta_2^{*}&=&\min\left(1, \max\left(0, -\frac{ m \lambda_2}{m+d_2}\right)\right).
\end{eqnarray}

\begin{theorem}
An optimal control $(\eta_1^*, \eta_2^*) \in U_5$ for the system \eqref{NHP_2} that maximizes the objective functional $J_2$ is characterized by \eqref{NHP2_11}-\eqref{NHP2_22}.
\end{theorem}

\subsection{Model 6: $\boldsymbol{G_1(f)=f^{\frac{3}{2}}, G_2(m)=m^{\frac{3}{2}}}$}
Consider an alternative model of the harvesting given by,
\begin{equation}
\label{NHP_3}
\begin{split}
& \frac{df}{dt}=\frac{1}{2}fm\beta L-\delta f-\eta_1 f^{\frac{3}{2}},\\
& \frac{dm}{dt}=\frac{1}{2}fm\beta L-\delta m-\eta_2 m^{\frac{3}{2}}.
\end{split}
\end{equation}

\subsubsection{Global stability}
The global stability of the trivial equilibrium for model \eqref{NHP_3} is proven in the following theorem.
\begin{theorem}
\label{NHP3_global}
The trivial equilibrium $(0,0)$ for \eqref{NHP_3} is globally asymptotically stable if $\beta K<2\delta$.
\end{theorem}
\begin{proof}
Again, consider the Lyapunov function $V=f+m$. It is left to show that $\frac{dV}{dt}<0$ for all $(f,m) \not =(0,0)$. Consider 
\beq
\begin{split}
\frac{dV}{dt}&=\frac{df}{dt}+\frac{dm}{dt}\\
&=fm\beta L-(\delta+\eta_1 \sqrt{f})f-(\delta+\eta_2 \sqrt{m})m\\
&\leq fm\beta -(\delta+\eta_1 \sqrt{f})f-(\delta+\eta_2 \sqrt{m})m, \quad \mathrm{since} \quad L \leq 1 \\
&\leq fm\beta-\delta f-\delta m\\
&\leq (\beta-\frac{\delta}{m}-\frac{\delta}{f}) fm\\
&\leq  (\beta-\frac{\delta}{K}-\frac{\delta}{K}) fm.\\
\end{split}
\eeq
It is enough to show $\beta-\frac{\delta}{K}-\frac{\delta}{K}<0$. By direct calculation, we have $\beta K<2\delta$, which proves theorem \eqref{NHP3_global}.
\end{proof}

\subsubsection{Optimal control}
The parameters $\eta_1$ and $\eta_2$ are assumed to be time dependent and define the objective function
\beq
  J_6(\eta_{1},\eta_{2})= \int^{T}_{0} -(f+m)- \frac{1}{2}(\eta_{1}^{2} +\eta_{2}^{2} ) dt
\eeq
subject to \eqref{NHP_3} and with initial conditions $f(t_{0})=f_{0}, m(t_{0})=m_0$. The optimal controls are sought within the set $U_6$ where
\beq
 U_6= \{ (\eta_{1},\eta_{2})~|~\eta_{i} \ \mbox{measurable}, \  0 \leq \eta_{1} \leq 1, 0 \leq \eta_{2} \leq 1, \  t \in [0,T], \ \forall T\}.
\eeq
The goal is obtain optimal $(\eta_{1}^{*}, \eta_{2}^{*})$ such that
\beq
\begin{split}
  J_6(\eta_{1}^{*}, \eta_{2}^{*}) &= \underset{(\eta_{1},\eta_{2})}{\max} \int^{T}_{0} -(f+m) - \frac{1}{2}(\eta_1^2+\eta_{2}^{2})  dt \\
  & = \underset{(\eta_{1},\eta_{2})}{\min} \int^{T}_{0} f+m + \frac{1}{2}(\eta_1^2+\eta_{2}^{2})  dt.\\
 \end{split}
\eeq
The following existence theorem is given.

\begin{theorem}
\label{PH3_thm_oct}
Consider the optimal control problem \eqref{HP_3}. There exists $(\eta_{1}^{*}, \eta_{2}^{*}) \in U_6$ such that
\begin{equation}
  J_6(\eta_{1}^{*}, \eta_{2}^{*}) = \underset{(\eta_{1},\eta_{2})}{\max} \int^{T}_{0} -(f+m) - \frac{1}{2}(\eta_1^2+\eta_{2}^{2})  dt.
\end{equation}
\end{theorem}

\begin{proof}
The proof is similar to Theorem \ref{thm_NPH1_oct} and is omitted for brevity.
\end{proof}

As in the previous sections, the Hamiltonian for $J_6$ is considered, that is,
\beq
\label{HP3_ham}
H_6=-(f+m) - \frac{1}{2}(\eta_1^2+\eta_{2}^{2}) +\lambda_1 f'+\lambda_2 m'.
\eeq
Again, we establish differential equations for the adjoint,
\beq
\begin{split}
\lambda_1'(t) =& 1-\lambda_1 [\frac{m \beta }{2} (1-\frac{f+m}{K})-\frac{f m \beta}{2K}-\delta-\frac{3}{2} \eta_1 f^{\frac{1}{2}}] \\
& -\lambda_2 [\frac{ m \beta}{2} (1-\frac{f+m}{K})-\frac{ f m \beta}{2K}],\\
\lambda_2'(t)= &1-\lambda_1 [\frac{f \beta }{2} (1-\frac{f+m}{K})-\frac{f m \beta}{2K}] \\
&-\lambda_2 [\frac{ f \beta}{2} (1-\frac{f+m}{K})-\frac{ f m \beta}{2K}-\delta-\frac{3}{2} \eta_2 m^{\frac{1}{2}}],
\end{split}
\eeq
with the transversality condition $\lambda_1(T)=\lambda_2(T)=0$.  The derivatives of the Hamiltonian are,
\beq
\begin{split}
& \frac{\partial H_6}{\partial \eta_1}=-\lambda_1 f^{\frac{3}{2}}-\eta_1,\\
& \frac{\partial H_6}{\partial \eta_2}=-\lambda_2 m^{\frac{3}{2}}-\eta_2.
\end{split}
\eeq
It is simple to verify that the optimal controls $\eta_1^*$ and $\eta_2^*$ are compactly given by
\begin{eqnarray}
\label{NHP3_1}
\eta_1^{*}&=&\min(1, \max(0, -\lambda_1 f^{\frac{3}{2}})), \\
\label{NHP3_2}
\eta_2^{*}&=&\min(1, \max(0, -\lambda_2 m^{\frac{3}{2}})).
\end{eqnarray}

\begin{theorem}
An optimal control $(\eta_1^*, \eta_2^*) \in U_6$ for the system \eqref{NHP_3} that maximizes the objective functional $J_6$ is characterized by \eqref{NHP3_1}-\eqref{NHP3_2}.
\end{theorem}

\end{document}